\documentclass{amsart}

\usepackage{amsthm} 
\usepackage{amssymb}
\usepackage{graphicx}									 
\usepackage{stmaryrd}                  
\usepackage{footmisc}                  
\usepackage{paralist}
\usepackage{wrapfig}
\usepackage[draft]{pdfcomment}
\usepackage[ngerman,british]{babel}
\usepackage{color}

\newcommand{\plustimes}{\mathpalette\plustimesinner\relax}
\newcommand{\plustimesinner}[2]{%
  \mathbin{\vphantom{+}\ooalign{$#1+$\cr\hidewidth$#1\times$\hidewidth\cr}}%
}

\usepackage{enumitem}
\usepackage{fancyvrb}

\renewcommand{\theenumi}{\roman{enumi}}
\newtheorem*{thm-plain}{Theorem}

\newtheorem{theorem}{Theorem}
\newtheorem{lemma}[theorem]{Lemma}
\newtheorem{corollary}[theorem]{Corollary}
\newtheorem{proposition}[theorem]{Proposition}
\newtheorem*{qst}{Question}

\def\bysame{\leavevmode\hbox to3em{\hrulefill}\thinspace}

\theoremstyle{definition}
\newtheorem{definition}{Definition}

\theoremstyle{remark}
\newtheorem{remark}{Remark}

\theoremstyle{remark}

\newcommand{\C}{\mathbb{C}}
\newcommand{\R}{\mathbb{R}}

\newcommand{\Z}{\mathbb{Z}}

\newcommand{\alt}{\mathfrak{A}}
\newcommand{\sym}{\mathfrak{S}}
\newcommand{\dih}{\mathfrak{D}}
\newcommand{\cyc}{\mathfrak{C}}

\newcommand{\SOr}{\mathrm{SO}}
\newcommand{\SU}{\mathrm{SU}}
\newcommand{\U}{\mathrm{U}}
\newcommand{\PU}{\mathrm{PU}}
\newcommand{\PSU}{\mathrm{PSU}}
\newcommand{\PSL}{\mathrm{PSL}}
\newcommand{\PSp}{\mathrm{PSp}}
\newcommand{\SL}{\mathrm{SL}}
\newcommand{\GL}{\mathrm{GL}}
\newcommand{\PGL}{\mathrm{PGL}}
\newcommand{\AG}{\mathrm{AG}}
\newcommand{\Or}{\mathrm{O}}

\newcommand{\A}{\mathrm{A}}
\newcommand{\BC}{\mathrm{BC}}
\newcommand{\D}{\mathrm{D}}
\newcommand{\Rr}{\mathrm{R}}
\newcommand{\Qq}{\mathrm{Q}}
\newcommand{\Pp}{\mathrm{P}}
\newcommand{\Ss}{\mathrm{S}}
\newcommand{\Tt}{\mathrm{T}}

\newcommand{\E}{\mathrm{E}}
\newcommand{\F}{\mathrm{F}}

\newcommand{\Hn}{\mathrm{H}}
\newcommand{\I}{\mathrm{I}}

\newcommand{\Wp}{W^{\scriptscriptstyle+}}
\newcommand{\Fp}{F^{\scriptscriptstyle+}}
\newcommand{\Dp}{D^{\scriptscriptstyle+}}
\newcommand{\Wt}{W^{\scriptscriptstyle\times}}

\newcommand{\Mt}{M^{\scriptscriptstyle\times}}
\newcommand{\Gt}{G^{\scriptscriptstyle\times}}
\newcommand{\Gs}{G^{\scriptscriptstyle\plustimes}}
\newcommand{\At}{A^{\scriptscriptstyle\times}}
\newcommand{\As}{A^{\scriptscriptstyle\plustimes}}
\newcommand{\Ws}{W^{\scriptscriptstyle\plustimes}}

\newcommand{\subgr}{<}
\newcommand{\subgreq}{<}

\newcommand{\To}{\rightarrow}

\newcommand{\MTo}{\mapsto}

\title[Classification of finite reflection-rotation groups]{Classification of finite groups generated by\\ reflections and rotations}
\author{Christian Lange and Marina A. Mikha\^ilova}
\thanks{The first named author has been supported by the German Academic Exchange Service (DAAD)}
\address{Christian Lange, Mathematisches Institut der Universit\"at zu K\"oln, Weyertal 86-90, 50931 K\"oln, Germany}
\email{clange@math.uni-koeln.de}
\address{Marina A. Mikha\^ilova, Moscow State Pedagogical University, 1/1 Malaya Pirogovskaya , Moscow, 119991, Russian Federation}

\begin{document}

\begin{abstract} We classify finite groups generated by orthogonal transformations in a finite-dimensional Euclidean space whose fixed point subspace has codimension one or two. These groups naturally arise in the study of the quotient of a Euclidean space by a finite orthogonal group and hence in the theory of orbifolds.
\end{abstract}
\maketitle

\section{Introduction}
\label{sec:Pse_Introduction}

A \emph{finite reflection group} is a finite group generated by \emph{reflections} in a finite-dimensional Euclidean space, i.e. by orthogonal transformations of this space whose fixed point subspace has codimension \emph{one}. Analogously, we say that a finite group is a \emph{finite rotation group}, if it is generated by orthogonal pseudoreflections in a finite-dimensional Euclidean space, i.e. by orthogonal transformations of this space whose fixed point subspace has codimension \emph{two}\footnote{The term ``pseudoreflection'' for a linear transformation in a finite-dimensional real vector space whose fixed point subspace has codimension two was introduced in \cite{Maerchik}. One should however note that some authors, e.g. Bourbaki, use it with a different meaning.}. Since an orthogonal pseudoreflection necessarily rotates the two-dimensional complement of its fixed point subspace, we also call it a \emph{rotation}. A \emph{finite reflection-rotation group} is then a finite group generated by reflections and rotations in a finite-dimensional Euclidean space. From now on the specification \emph{finite} for reflection-rotation groups is understood.

The quotient of a finite-dimensional Euclidean space by a finite group generated by orthogonal transformations in this space inherits many structures from the initial space, e.g. a topology, a metric and a piecewise linear structure. The question when it is a manifold with respect to one of these structures arises naturally, for example in the theory of orbifolds as pointed out by Davis \cite{MR2883685}. If it is a manifold, then, depending on the specific category, it is true or at least almost true that the acting group is a rotation group (cf. \cite{Shvartsman,MR1217530}, \cite{Styrt}). In the topological category a counterexample to this statement is given by the \emph{binary icosahedral group} \cite[p.~9]{MR2883685}. Conversely, it has been verified in many cases that the quotient is homeomorphic to the initial space, if the acting group is a rotation group \cite{Mikhailova}. Moreover, for quotients that are manifolds with boundary also general reflection-rotation groups occur. It is therefore desirable to have a complete classification of reflection-rotation groups. 

Large classes of reflection-rotation groups are real reflection groups, their orientation preserving subgroups and unitary reflection groups considered as real groups. In this paper we obtain a complete classification of reflection-rotation groups. It turns out that an irreducible reflection-rotation group essentially belongs to one of the mentioned classes, at least in high dimensions. Nevertheless, the fact that reducible reflection-rotation groups, in contrast to reducible reflection groups, in general do not split as products of irreducible components gives rise to many more nontrivial examples.

Two reflections in a reflection group generate a dihedral group which is characterized by its order, or equivalently by the angle between the two corresponding reflection hyperplanes. In 1933 Coxeter classified reflection groups by determining the possible configurations of reflections in such a group \cite{MR1503182}. This information, i.e. the dihedral groups defined by pairs of certain generating reflections, is encoded in the corresponding Coxeter diagram. Similarly, two rotations in a rotation group generate a rotation group in dimension two, three or four and all groups that arise in this way are known explicitly. However, an approach to the classification of rotation groups similar to the one for reflection groups, albeit conceivable, seems to be unpractical. Instead, we follow an approach outlined in \cite{Mikhailova} that has already been carried out partially. Classifications of several subclasses of rotation groups are treated in \cite{Maerchik,MR608821,Mikhailova2}. From these papers and from results by Brauer, Huffman and Wales \cite{MR0206088,MR0401937,MR0401936,MR928600} a complete classification of rotation groups can be obtained. However, the latter results have not been worked out yet with regard to such a classification and in fact new examples appear in these cases. The largest irreducible rotation group among them occurs in dimension $8$ and is connected with some grading of the simple Lie algebra $\mathfrak{so}_8$ (cf. Theorem \ref{thm:clas_red_psg}, $(v)$, $3.$). It is an extension of the alternating group on $8$ letters by a nonabelian group of order $2^7$ and contains many other exceptional rotation groups as subgroups (cf. Section \ref{sub:excep_subgroups}). The other irreducible examples and the building blocks of the reducible examples appear as subgroups in the normalizers of reflection groups (cf. Theorem \ref{thm:clas_red_psg}, the groups $\Ws(\A_5)$ and $\Ws(\E_6)$, and Theorem \ref{thm:reducible_pairs}, $(xvi)$, type $\A_5$ and $\E_6$). Other interesting reducible rotation groups that have not been studied in \cite{Mikhailova} occur in the product of two copies of a reflection group $W$ of type $\Hn_3$ or $\Hn_4$ due to the existence of outer automorphisms of $W$ that map reflections onto reflections but cannot be realized through conjugation by elements in its normalizer (cf. Section \ref{sub:reducible_examples}).

In this paper we survey the existing parts of a classification of rotation groups and extend it to a complete classification. Besides, we generalize the proofs as to also yield a classification of reflection-rotation groups. Finally, in the last section, we discuss a question on isotropy groups of reflection-rotation groups.

In forthcoming papers we will, based on the obtained classification, characterize reflection-rotation groups in terms of the corresponding quotient spaces endowed with additional structures and show that the binary icosahedral group is essentially the only counterexample to this characterization in the topological category (cf. \cite{Lange3,Lange,Lange4}).
\newline
\linebreak
\emph{Acknowledgements.} We are thankful to \`Ernest B. Vinberg for many useful suggestions and remarks. His comments and hints to \cite{MR1349140,MR0379748,MR1191170} and \cite{MR0237660} helped to simplify the considerations in Section \ref{sec:group_L_sec} considerably. The first named author would like to thank the second named author and \`Ernest B. Vinberg for their friendliness during his stay in Moscow. He would also like to thank David Wales for answering his questions and Franz-Peter Heider for helpful suggestions. Finally, he is grateful to Alexander Lytchak for his advice and his help with translating Russian papers.

\section{Notations}
\label{sec:not_ter}

We denote the cyclic group of order $n$ and the dihedral group of order $2n$ by $\cyc_n$ and $\dih_n$, respectively. We denote the symmetric and alternating group on $n$ letters by $\sym_n$ and $\alt_n$. For a finite field of order $q$ we write $\mathbb{F}_q$. Classical Lie groups are denoted like $\mathrm{SO}_n$ and $\U_n$. The classical groups over finite fields are denoted like $\SL_n(q)=\SL_n(\mathbb{F}_q)$ as in \cite{MR827219}. For a finite subgroup $G<\Or_n$ we denote its orientation preserving subgroup as $G^{\scriptstyle+}<\SOr_n$. In particular, we write $\Wp$ for the orientation preserving subgroup of a reflection group $W$ of a certain type. A list of all groups we are going to introduce can be found in the appendix.

\section{Strategy and Results}
\label{sec:strategy}

We first classify rotation groups and afterwards reduce the classification of reflection-rotation groups to the classification of reflection groups and the classification of rotation groups. A rotation group preserves the orientation. Conversely, all finite subgroups of $\SOr_2$ and $\SOr_3$ are rotation groups. The finite subgroups of $\SOr_4$ are listed in \cite{MR0169108} and the rotation groups among them can be singled out. The classifications of \emph{irreducible} and \emph{reducible} rotation groups have to be treated separately since a reducible rotation group in general does not split as a product of irreducible components. If the complexification of an irreducible rotation group is reducible then this group preserves a complex structure and is thus a unitary reflection group considered as a real group. Otherwise it is called \emph{absolutely irreducible} and we make another case differentiation. Depending on whether there exists a decomposition of the underlying vector space into nontrivial subspaces that are interchanged by $G$, a so-called \emph{system of imprimitivity}, or not, the group is either called \emph{imprimitive} or \emph{primitive}. For an imprimitive irreducible rotation group the subspaces forming a system of imprimitivity are either all one- or two-dimensional. In the first case the group is called \emph{monomial}.

The classification of imprimitive irreducible rotation groups and of reducible rotation groups that cannot be written as a product of irreducible components is essentially contained in \cite{MR608821} and \cite{Mikhailova2}, respectively. The classification of primitive rotation groups in dimension four is treated in \cite{MR0169108}. The classification in higher dimensions can be obtained from results by Brauer \cite{MR0206088} and by Huffman and Wales \cite{MR0401937,MR0401936} in which they classify finite quasiprimitive irreducible linear groups over the complex numbers that contain an element whose fixed point subspace has complex codimension two. However, certain cases in these papers are only implicitly described, have not been worked out in detail with regard to rotation groups before and are the source of the new examples.

For a monomial group $G$ we denote its diagonal subgroup, i.e. the set of all transformations that act trivially on its system of imprimitivity, by $D(G)$. Apart from the two families of orientation preserving subgroups of the reflection groups $W(\BC_n)$ and $W(\D_n)$, there are four monomial rotation groups $M_5$, $M_6$, $M_7$ and $M_8$ given as semidirect products of the diagonal subgroup of $W(\D_n)$ and a permutation group $H$, and two exceptional subgroups $M^p_7$ and $M^p_8$ of $M_7$ and $M_8$, respectively. 

There is a class of imprimitive unitary reflection groups, denoted by $G(m,p,n)\subgr\U_n$, which is defined to be the semidirect product of
\[
		A(m,p,n) := \left\{(\theta_1,\ldots,\theta_n)\in \mu_m^n | (\theta_1\dots \theta_n)^{m/p}=1 \right\}
\]
with the symmetric group $\sym_n$, where $\mu_m \subgr \C^*$ is the cyclic subgroup of $m$-th roots of unity and $p$ is a factor of $m$. The only other imprimitive irreducible rotation groups are extensions of $G(m,1,n)$ and $G(2m,2,n)$ by a rotation $r$ that conjugates two coordinates, i.e. $r(z_1,z_2,z_3\ldots,z_n)=(\overline{z}_1,\overline{z}_2,z_3\ldots,z_n)$. We denote these groups by $\Gs(km,k,n)$, $k=1,2$.

Apart from the primitive rotation groups that are either orientation preserving subgroups of real reflection groups or unitary reflection groups considered as real groups, there are five primitive rotation groups $\Ws$ obtained by extending the orientation preserving subgroup $\Wp$ of a real reflection group $W$ by a normalizing rotation, two exceptional primitive rotation groups in dimensions five and six isomorphic to $\alt_5$ and $\PSU_2(7)$, respectively, and a primitive rotation group in dimension eight, which is generated by $M_8$ and another rotation.

The rotation groups listed in Theorem \ref{thm:clas_red_psg}, $(v)$ are generated by rotations of order $2$. A rotation group $G<\SOr_n$ with this property defines a configuration $\mathfrak{P}=\{\sigma_i\}_{i\in I}$ of $2$-planes in $\R^n$ given by the complements of the fixed point subspaces of the involutive rotations in $G$ such that $r_{\sigma}(\mathfrak{P})=\mathfrak{P}$ holds for all $\sigma \in \mathfrak{P}$ where $r_{\sigma}$ is the rotation of order $2$ defined by $\sigma$. We call such a configuration a \emph{plane system} and denote the generated rotation group by $M(\mathfrak{P})$.

\begin{theorem}\label{thm:clas_red_psg}
Every irreducible rotation group occurs, up to conjugation, in precisely one of the following cases
\begin{enumerate}
\item Orientation preserving subgroups $\Wp$ of irreducible real reflection groups $W$ (cf. Section \ref{sub:reflection-groups}).
\item Irreducible unitary reflection groups $G<\U_n$, $n\geq 2$, that are not the complexification of a real reflection group, considered as real groups $G<\SOr_{2n}$ (cf. Section \ref{sub:unitary}).
\item The imprimitive rotation groups $\Gs(km,k,l)\subgr\SOr_{n}$ for $n=2l>4$, $k\in\{1,2\}$ and $km\geq3$ (cf. Proposition \ref{prp:class_imprimit_psg}).
\item The unique extensions $\Ws$ of $\Wp$ by a normalizing rotation for real
reflection groups $W$ of type $\A_4$, $\D_4$, $\F_4$, $\A_5$ and $\E_6$ (cf. Section \ref{sub:reflection-groups}). These groups are primitive. 
\item The following rotation groups which can be realized as $M(\mathfrak{P})$ for a plane system $\mathfrak{P}$ of type $\Pp_5$, $\Pp_6$, $\Pp_7$, $\Pp_8$, $\Qq_7$, $\Qq_8$, $\Ss_5$, $\Rr_6$ or $\Tt_8$ and which only contain rotations of order $2$ (cf. Section \ref{sub:excep_subgroups}), namely
\begin{enumerate}
\item the monomial rotation groups $M_5$, $M_6$, $M_7$, $M_8$ and $M^p_7$ and $M^p_8$ (cf. Proposition \ref{prp:class_monomial_psg}).
\item the primitive rotation groups $R_5(\alt_5)$ and $R_6(\PSL_2(7))$ given as the image of the unique irreducible representations of $\alt_5$ in $\SOr_5$ and of $\PSL_2(7)$ in $\SOr_6$ (cf. Section \ref{sub:primitive}).
\item a primitive rotation group in $\SOr_8$ isomorphic to an extension of $\alt_8$ by a nonabelian group of order $2^7$ (cf. Section \ref{sec:group_L_sec}).
\end{enumerate}
\item The remaining rotation groups in $\SOr_4$, i.e. an infinite family of imprimitive rotation groups described in Proposition \ref{prp:class_imprimit_psg4} and $3$ individual- and $6$ infinite families of primitive rotation groups listed in Table \ref{tab:rot4}, Section \ref{sub:ps_low_dim}.
\end{enumerate}
\end{theorem}

For a real reflection group $W$ we denote by $\Wt$ its unique extension by a normalizing rotation, provided such exists, i.e. $\Wt=\left\langle \Ws,W \right\rangle$. For a monomial rotation group $M$ we denote by $\Mt$ its extension by a coordinate reflection. Finally, for an imprimitive rotation group of type $G(km,k,l)$ let $\Gt (km,k,l)$ be its extension by a reflection $s$ of the form $s(z_1,\ldots, z_l)=(\overline{z}_1,z_2,\ldots, z_l)$.
\begin{theorem}\label{thm:irr_re-ro_group}
Every irreducible reflection-rotation group either appears in Theorem \ref{thm:clas_red_psg} or it contains a reflection and occurs, up to conjugation, in one of the following cases
\begin{enumerate}
\item Irreducible real reflection groups $W$ (cf. Section \ref{sub:reflection-groups}).
\item The groups $\Wt$ generated by a reflection group $W$ of type $\A_4$, $\D_4$, $\F_4$, $\A_5$ or $\E_6$ and a normalizing rotation (cf. Section \ref{sub:reflection-groups}).
\item The monomial groups $\Mt$ of type $\D_n$, $\Pp_5$, $\Pp_6$, $\Pp_7$, $\Pp_8$, i.e. $\Mt(\D_n):=D(W(\BC_n)) \rtimes \alt_n$, $\Mt_5$, $\Mt_6$, $\Mt_7$ and $\Mt_8$ (cf. Section \ref{sub:monomial_examples}).
\item The imprimitive groups $\Gt (km,k,l)\subgr\SOr_{n}$ with $n=2l$, $k=1,2$ and $km\geq3$ (cf. Section \ref{sub:imprimitive_example}).
\end{enumerate}
\end{theorem}

Let $G\subgr \Or_n$ be an arbitrary reflection-rotation group and let $\R^n=V_1 \oplus \ldots \oplus V_k$ be a decomposition into irreducible components. For each $i \in I = \{1,\ldots,k\}$ we denote the projection of $G$ to $\Or(V_i)$ by $\pi_i$ and set $G_i = \pi_i(G)$. We distinguish two kinds of rotations in $G$ (cf. \cite{Mikhailova2}).

\begin{definition}\label{dfn:first_second_kind}
A rotation $g\in G$ is called a rotation of the
\begin{enumerate}
	\item \emph{first kind}, if for some $i_0 \in I$, $\pi_{i_0}(g)$ is a rotation in $V_{i_0}$ and $\pi_i(g)$ is the identity on $V_{i}$ for all $i\neq i_0$.
	\item \emph{second kind}, if for some $i_1,i_2 \in I$, $i_1 \neq i_2$, $\pi_{i_1}(g)$ and $\pi_{i_2}(g)$ are reflections in $V_{i_1}$ and $V_{i_2}$, respectively, and $\pi_i(g)$ is the identity for all $i\neq i_1,i_2$.
\end{enumerate}
\end{definition}

Let $H$ be the normal subgroup of $G$ generated by rotations of the first kind, let $F$ be the normal subgroup of $G$ generated by reflections and rotations of the second kind and set $H_i=\pi_i(H)$ and $F_i=\pi_i(F)$. Then $H_i$ is a rotation group, $F_i$ is a reflection group and both are normal subgroups of $G_i$. In order to classify all reflection-rotation groups we first describe the possible triples $(G_i,H_i,F_i)$ and then the ways how a reflection-rotation group can be recovered from a collection of such triples. Notice that $G_i$ is generated by $H_i$ and $F_i$. Hence, depending on whether $F_i$ is trivial or not, $G_i$ either appears in Theorem \ref{thm:clas_red_psg} or in Theorem \ref{thm:irr_re-ro_group}. Recall that reflections $s_1,\ldots,s_l$ whose corresponding fixed-point hyperplanes are the walls of a chamber of a reflection group $W$ generate $W$ and that $(W,S)$ is a Coxeter system for $S=\{s_1,\ldots,s_l\}$ (cf. \cite[p.~10, p.~23]{MR1066460}). We refer to the reflections $s_1,\ldots,s_l$ as \emph{simple reflections} (cf. \cite[p.~10]{MR1066460}).

\begin{theorem} \label{thm:reducible_pairs} Let $G\subgr\Or_n$ be a reflection-rotation group. Then, for each $i$, either $G_i=H_i$ is an irreducible rotation group or $F_i$ is nontrivial and a set of simple reflections generating $F_i$ projects onto a set $\overline{S}\subset G_i/H_i$ for which $(G_i/H_i,\overline{S})$ is a Coxeter system. In the second case the quadruple $(G_i,H_i,F_i,\Gamma_i)$ occurs, up to conjugation, in one of the following cases where $\Gamma_i$ denotes the Coxeter diagram of $G_i/H_i$.
\begin{enumerate}
\item $(\Mt,M,D(\Mt),\circ)$ for $M=M_5, M_6,M_7, M_8, M(\D_n)=\Wp(\D_n)$.
\item $(\Gt (km,k,l),\Gs(km,k,l),D(\Gt (km,k,l)),\circ)$ for $km\geq 3$ and $n=2l$.
\item $(\Gt (2m,1,l),\Gs(2m,2,l),D(\Gt (2m,1,l)),\circ \ \ \ \circ)$ for $m\geq 2$ and $n=2l$.
\item $(W,\{e\},W,\Gamma(W))$ for any irreducible reflection group $W$.
\item $(W,\Wp,W,\circ)$ for any irreducible reflection group $W$.
\item $(W(\A_3),\Wp(\A_1 \times \A_1 \times \A_1), W(\A_3),\circ-\circ)$
\item $(W(\BC_n),D(\Wp(\BC_n)),W(\BC_n),\Gamma(\A_{n-1}\times \A_1)=\circ-\circ-\dots \circ \ \ \ \circ)$
\item $(W(\BC_n),\Wp(\D_n),W(\BC_n),\circ \ \ \ \circ)$
\item $(W(\BC_4),\Gs(4,2,2),W(\BC_4),\circ - \circ \ \ \ \circ)$
\item $(W(\D_n),D(W(\D_n)),W(\D_n),\Gamma(\A_{n-1})=\circ-\circ-\dots \circ)$
\item $(W(\D_4),\Gs(4,2,2),W(\D_4),\circ - \circ)$
\item $(W(\I_2(km)),\Wp(\I_2(m)),W(\I_2(km)),\circ\stackrel{k}{-}\circ)$ for $m,k\geq2$.
\item $(W(\F_4),\Gs(4,2,2),W(\F_4),\circ-\circ \ \ \ \circ-\circ)$
\item $(W(\F_4),\Wp(\D_4),W(\F_4),\circ-\circ \ \ \ \circ)$
\item $(W(\F_4),\Ws(\D_4),W(\F_4),\circ \ \ \ \circ)$
\item $(\Wt,\Ws,W,\circ)$ for a reflection group $W$ of type $\A_4$, $\D_4$, $\F_4$, $\A_5$ or $\E_6$.
\item $(\Wt(\D_4),\Wp(\D_4),W(\D_4),\circ-\circ)$ (, but $H_i \neq \Fp_i$, cf. Proposition \ref{prp:red3}.)
\end{enumerate}
For each quadruple $(G_{rr},M,W,\Gamma)$ occurring in this list every reflection in $G_{rr}$ is contained in $W$. The group $W$ is reducible in the cases $(i)$ to $(iii)$, irreducible with $W=G_{rr}$ in the cases $(iv)$ to $(xv)$ and irreducible with $W\neq G_{rr}$ in the cases $(xvi)$ and $(xvii)$.
\end{theorem}

\begin{remark} The preceding theorem is actually a classification of triples $(G_{rr},M,W)$ where $G_{rr}<\Or_n$ is a finite irreducible group generated by a rotation group $M$ and a reflection group $W$ which are both normal subgroups of $G_{rr}$ such that the following additional condition holds. If $hs$ is a reflection for some $h \in M$ and some reflection $s \in W$, then it is contained in $W$ (cf. Section \ref{sec:class_reducible}).
\end{remark}

Assume that the family of triples $\{(G_i,H_i,F_i)\}_{i\in I}$, $I=\{1,\ldots,k\}$, is induced by a reflection-rotation group. The reflections in $\tilde{G}=G_1/H_1\times\dots\times G_k/H_k$ are the cosets of the reflections in $F_1\times\dots\times F_k$. We call two such reflections $\overline{s}_1 \in G_i/H_i$ and $\overline{s}_2 \in G_j/H_j$ for $i\neq j$ \emph{related}, if $s_1 \notin G$ and if there exists a rotation of the second kind $h\in G$ such that $s_1=\pi_i(h)$ and $s_2=\pi_j(h)$. Relatedness of reflections defines an equivalence relation on the set of reflections in $\tilde{G}$. This equivalence relation induces an equivalence relation on the set of irreducible components of $\tilde{G}$ and on the set of connected components of its Coxeter diagram such that equivalence classes of nontrivial irreducible components, i.e. of those whose Coxeter diagram is not an isolated vertex, consist of two isomorphic components that belong to different $G_i/H_i$ and are isomorphic via an isomorphism induced by relatedness of reflections.

Conversely, given such data one first obtains an equivalence relation on the set of reflections contained in $G_1/H_1\times\dots\times G_k/H_k$ and then a reflection-rotation group $G\subgr G_1 \times \dots \times G_k$ generated by $H$, the rotations $s_1s_2$ for reflections $s_1\in F_i$ and $s_2\in F_j$, $i\neq j$, whose cosets $\overline{s}_1$ and $\overline{s}_2$ are equivalent, and the reflections $s\in F_i$ whose cosets are not equivalent to any other coset of a reflection.

In fact, these assignments are inverse to each other.
\begin{theorem}\label{thm:clas_red_prg}
Reflection-rotation groups are in one-to-one correspondence with families of triples occurring in Theorem \ref{thm:reducible_pairs}, $\{(G_i,H_i,F_i)\}_{i\in I}$, with an equivalence relation on the set of irreducible components of $\tilde{G}=G_1/H_1\times\dots\times G_k/H_k$ such that
\begin{enumerate}
\item the elements of an equivalence class belong to pairwise different $G_i/H_i$,
\item each $G_i/H_i$ contains at most one trivial irreducible component that is not equivalent to another component,
\item equivalence classes of nontrivial irreducible components contain precisely two isomorphic components
\end{enumerate}
together with isomorphisms between the equivalent nontrivial irreducible components that map reflections onto reflections. A reflection-rotation group corresponding to such a set of data contains a reflection, if and only if there exists an equivalence class consisting of a single trivial component.
\end{theorem}

Notice that different isomorphisms between the irreducible components in general yield nonconjugate reflection-rotation groups (cf. Section \ref{sub:reducible_examples}).

\section{Examples and properties}
\label{sec:pseudoreflection_groups}
In this section we describe several classes of reflection-rotation groups and discuss some of their properties.

\subsection{Real reflections groups}
\label{sub:reflection-groups}
A real reflection group $W$ is a finite subgroup of an orthogonal group $\Or_n$ generated by reflections, i.e. by orthogonal transformations whose fixed point subspace has codimension one. Irreducible reflection groups are classified and the types of the occurring groups are denoted as $\A_n$, $\BC_n$, $\D_n$, $\E_6$, $\E_7$, $\E_8$, $\F_4$, $\Hn_3$, $\Hn_4$ and $\I_2(p)$ for $p\geq 3$ \cite{MR1066460}. Every reflection group splits as a direct product of irreducible reflection groups. Since the composition of two distinct reflections is a rotation and since all compositions of pairs of reflections in a reflection group $W$ generate the orientation preserving subgroup $\Wp$ of $W$, this subgroup $\Wp$ is always a rotation group. There is another way to construct a rotation group from a reflection group. If there exists a rotation $h\in \SOr_n\backslash W$ that normalizes $W$, then $h$ also normalizes $\Wp$ and the group $\Ws=\left\langle \Wp,h\right\rangle$ is again a rotation group. Now we specify the cases in which new examples arise this way.

\begin{lemma}\label{lem:existence_chamber}
Let $W\subgr \Or_n$ be a reflection group and suppose $h \in \SOr_n\backslash W$ is a rotation that normalizes $W$. If $\left\langle W,h\right\rangle$ is not a reflection group, then there exists a chamber $C$ of $W$ such that $hC=C$.
\end{lemma}
\begin{proof} Since $h$ normalizes $W$, it interchanges the chambers of $W$ \cite[p.~23]{MR1066460}. If $U = \mathrm{Fix}(h)$ intersects the interior of some chamber $C$ of $W$, then we have $hC=C$. Otherwise $U$ would be contained in a hyperplane corresponding to some reflection $s$ in $W$. But then $sh$ would be a reflection and thus $\left\langle W,h\right\rangle=\left\langle W,sh\right\rangle$ would be a reflection group. This contradicts our assumption and so the claim follows.
  \end{proof}

\begin{lemma}\label{lem:implication_of_fixed_chamber}
Let $W\subgr \Or_n$ be an irreducible reflection group, let $C$ be a chamber of $W$ and suppose $h\in \SOr_n$ is a rotation such that $hC=C$. Then $W$ has type $\A_4$, $\D_4$, $\F_4$, $\A_5$ or $\E_6$. Moreover, in the case of type $\A_4$, $\F_4$, $\A_5$ and $\E_6$ such a rotation $h$ is unique. In the case of type $\D_4$ there exist two such rotations which have order $3$ and are inverse to each other.
\end{lemma}
\begin{proof} Because of $hC=C$, the rotation $h$ permutes the walls of the chamber $C$ and thus corresponds to an automorphism of the Coxeter diagram of $W$ \cite[p.~29]{MR1066460}. Since the fixed point subspace of $h$ has codimension two, we conclude that only the types $\A_4$, $\D_4$, $\F_4$, $\A_5$ and $\E_6$ can occur for $W$. The additional claims follow from the structure of the respective diagrams.
  \end{proof}

\begin{lemma}\label{prp:extended_psrg}
Let $W\subgr\SOr_n$ be an irreducible reflection group. Then there exists a rotation $h\in \SOr_n\backslash W$ that normalizes $W$ and $\Wp$ such that $\Ws=\left\langle \Wp,h\right\rangle$ is a rotation group which is not the orientation preserving subgroup of a reflection group if and only if $W$ has type $\A_4$, $\D_4$, $\F_4$, $\A_5$ or $\E_6$. In this case the extended group $\Ws$ is unique.
\end{lemma}
\begin{proof} The \emph{only if} direction is clear by the preceding two lemmas. Conversely, suppose that $W$ has type $\A_4$, $\D_4$, $\F_4$, $A_5$ or $\E_6$. In each case there exists a nontrivial automorphism of the Coxeter diagram of $W$. The vertices of this diagram correspond to a set $\Delta$ of outward normal vectors to the walls of a chamber of $W$ and the diagram automorphism corresponds to a permutation of $\Delta$ \cite[p.~29]{MR1066460}. Due to the fact that $\Delta$ is a basis of $\R^n$, this permutation can be extended to a linear transformation $h$ of $\R^n$. Since $h$ is induced by a diagram automorphism, it preserves the inner products of the vectors in $\Delta$ which are encoded in the Coxeter diagram of $W$. Hence the transformation $h$ is orthogonal, i.e. we have $h \in\Or_n$. Moreover, the structure of the Coxeter diagram of $W$ implies that the fixed point subspace of $h$ has codimension two and that the extension $\Ws=\left\langle \Wp,h\right\rangle$ obtained in this way is unique. Finally, it follows easily from the classification of reflection groups, e.g. by a counting argument, that $\Ws$ is not the orientation preserving subgroup of a reflection group in each of the cases $\A_4$, $\D_4$, $\F_4$, $\A_5$ and $\E_6$.
  \end{proof}

The next lemma will be needed later.

\begin{lemma}\label{lem:reflection_extension} For $n\geq5$ let $W\subgr \Or_n$ be an irreducible reflection group with orientation preserving subgroup $\Wp$. Assume that  $\left\langle W,\mathrm{-id}\right\rangle$ is not a reflection group. Then the group $G := \left\langle \Wp,\mathrm{-id}\right\rangle$ is a rotation group different from $\Wp$ if and only if $W$ has type $\E_6$.
\end{lemma}
\begin{proof} Assume that $G$ is a rotation group different from $\Wp$. Then there exists a rotation $h\in G \backslash \Wp$ that normalizes $W$. It follows from Lemma \ref{lem:existence_chamber}, Lemma \ref{lem:implication_of_fixed_chamber} and our assumption $n\geq5$ that $W$ has type $\A_5$ or $\E_6$. Since the inversion only preserves the orientation in even dimensions we conclude that $W$ has type $\E_6$.

Conversely, assume that $W$ has type $\E_6$ and let $C$ be any chamber of $W$. Since the inversion interchanges the chambers of $W$ and $W$ acts transitively on them \cite[p.~10]{MR1066460}, there exists some $w\in W$ such that $-wC=C$. The fact that the inversion is not contained in $W(\E_6)$ implies that the transformation $-w$ is nontrivial and thus induces a nontrivial automorphism of the Coxeter diagram of $W$. It follows from the structure of this diagram that $-w$ is a rotation which is why $G$ is a rotation group different from $\Wp$ as claimed.
  \end{proof}

Finally, we describe the groups $\Ws(\A_5)$ and $\Ws(\E_6)$ more explicitly.

\begin{proposition} The rotation groups $\Ws(\A_5)$ and $\Ws(\E_6)$ can be described as follows.
\label{prp:extended_reflection_groups}
\begin{enumerate}
\item $\Ws(\A_5) = \left\langle \Wp(\A_5),-s \right\rangle \cong \sym_6$ for any reflection $s\in W(\A_5)$.
\item $\Ws(\E_6) = \left\langle \Wp(\E_6),\mathrm{-id}\right\rangle \cong \PSU_2(4)\times \Z_2$.
\end{enumerate}
\end{proposition}
\begin{proof} For $(i)$ observe that $W(\A_5) \cong \sym_6$ has a trivial center and thus does not contain the inversion. It follows as in the proof of the preceding lemma that $\Wt(\A_5) = \left\langle W(\A_5), -\mathrm{id} \right\rangle $ and hence $\Ws(\A_5) = \left\langle \Wp(\A_5),-s \right\rangle \cong \sym_6$ as claimed. For $(ii)$ see the proof of the preceding lemma. 
  \end{proof}

\subsection{Unitary reflection groups}
\label{sub:unitary}
A unitary reflection group is a finite subgroup of some unitary group $\U_n$ generated by unitary reflections, i.e. by unitary transformations of finite order whose fixed point subspace has complex codimension one. A complete classification of such groups was first compiled by Shepard and Todd in 1954 \cite{MR0059914} and is described in \cite{MR2542964}. As in the real case, every unitary reflection group splits as a direct product of irreducible unitary reflection groups. The irreducible groups fall into two classes according to the following definition.

\begin{definition}
\label{dfn:imprimitive}
A finite subgroup $G\subgr \GL(V)$ is called \emph{imprimitive} if there exists a decomposition of the vector space $V$ into a direct sum of proper subspaces $V_1,\ldots, V_l$, a \emph{system of imprimitivity}, such that for any $g\in G$ and for any $i\in\{1,\ldots,l\}$ there exists a $j\in\{1,\ldots,l\}$ such that $\rho(g)(V_i)=V_j$. Otherwise the subgroup is called \emph{primitive}.
\end{definition}

The imprimitive irreducible unitary reflection groups can be constructed as follows (cf. \cite[Chapt. 2, p.~25]{MR2542964}). Let $\mu_m \subgr  \C^*$ be the cyclic subgroup of $m$-th roots of unity. For a factor $p$ of $m$ let
\[
		A(m,p,n) := \left\{(\theta_1,\ldots,\theta_n)\in \mu_m^n | (\theta_1\dots \theta_n)^{m/p}=1 \right\}
\]
and let $G(m,p,n)$ be the semidirect product of $A(m,p,n)$ with the symmetric group $\sym_n$. Then the natural realization of $G(m,p,n)$ in $\U_n$ is an imprimitive unitary reflection group and every imprimitive irreducible unitary reflection group is of this form. The following proposition holds \cite[Prop. 2.10, p. 26]{MR2542964}.

\begin{proposition} 
If $m>1$, then $G(m,p,n)$ is an imprimitive irreducible unitary reflection group except when $(m,p,n)=(2,2,2)$ in which case $G(m,p,n)$ is not irreducible.
\end{proposition}

A primitive unitary reflection group is either a cyclic group $\mu_n \subgr \U_1$, a symmetric group in $\U_n$ given as a complexified real reflection group of type $\A_n$, or one of $34$ primitive unitary reflection groups in dimension at most $8$ \cite[p.~138]{MR2542964}. Among the latter $34$ groups $19$ are two-dimensional. These groups decompose into $3$ families according to whether their image in $\PU_2\cong \SOr_3$ is a tetrahedral, an octahedral or an icosahedral group \cite[Chapt. 6 and Appendix D, Table 1]{MR2542964}. A collection $\mathfrak{L}$ of complex lines in $\C^n$ that is invariant under all reflections of order two defined by its lines is called a \emph{line system} and determines a unitary reflection group $W(\mathfrak{L})$ \cite[Chapt. 7]{MR2542964}. The remaining $15$ groups arise in this way. The occurring line systems are denoted as $\mathcal{E}_6$, $\mathcal{E}_7$, $\mathcal{E}_8$, $\mathcal{F}_4$ $\mathcal{H}_3$, $\mathcal{H}_4$, $\mathcal{J}_3^{(4)}$, $\mathcal{J}_3^{(5)}$, $\mathcal{K}_5$, $\mathcal{K}_6$, $\mathcal{L}_4$, $\mathcal{M}_3$, $\mathcal{N}_4$, $\mathcal{O}_4$ \cite[Thm. 8.29, p. 152 and Appendix D, Table 2]{MR2542964}, among them the complexifications of the root systems of real reflection groups of type $\E_6$, $\E_7$, $\E_8$, $\F_4$, $\Hn_3$ and $\Hn_4$.

Clearly, a unitary reflection group $G\subgr \U_n$ gives rise to a rotation group $G\subgr \SOr_{2n}$ when considered as a real group. Conversely, we have

\begin{lemma} \label{lem:abs_irr_or_complex} An irreducible rotation group is a unitary reflection group considered as a real group, if and only if it is not absolutely irreducible.
\end{lemma}
\begin{proof} The complexification of a complex group considered as a real group is reducible. In fact, it commutes with the idempotent product of the two complex structures and thus leaves its nontrivial $1$- and $(-1)$-eigenspace invariant. Conversely, let $G\subgr \SOr_n$ be an irreducible rotation group and suppose that $G$ is not absolutely irreducible. Then its complexification splits into more than one irreducible component, i.e.
\[
		V^{\C}=V_1\oplus\dots \oplus V_k
\]
for some $k\geq 2$. By restricting the scalars to the real numbers we recover two copies of the original representation and thus we have $k=2$ and $V=V_1^{\R}=V_2^{\R}$. Hence, $G$ is a unitary reflection group $G\subgr \U_m$ with $m=\dim(V_1)=n/2$ considered as a real group.
  \end{proof}

Moreover, we have

\begin{lemma} \label{lem:unitary_is_pseudo} A rotation group $G\subgr \SOr_{2n}$ that is a unitary reflection group $G\subgr \U_n$ considered as a real group is irreducible, if and only if $G\subgr \U_n$ is irreducible as a complex group and not the complexification of a real reflection group.
\end{lemma}
\begin{proof}If a group $G\subgr \U_n$ is irreducible over the real numbers, then it is also irreducible over the complex numbers and cannot be the complexification of a real group (cf. proof of Lemma \ref{lem:abs_irr_or_complex}). Conversely, assume that $G\subgr \U_n$ is an irreducible unitary reflection group which becomes reducible after restricting the scalars to the real numbers and let
\[
		\R^{2n}=V_1 \oplus \dots \oplus V_k
\]
be a corresponding decomposition into irreducible real subspaces for some $k\geq2$. Since the complex structure $J$ is preserved by $G$, the complex subspaces $V_1+ JV_1$ and $V_1 \cap JV_1$ are invariant under the action of $G$ and thus we have $\R^{2n}=V_1\oplus JV_1$, as $G$ is irreducible as a complex group by assumption. The fact that $G$ and $J$ commute moreover implies that the projection of $G$ to $\Or(V_1)$ is a real reflection group whose complexification is $G$.
  \end{proof}

We have the following criterion for an irreducible rotation group not to be induced by a unitary reflection group. In particular, it applies to the groups $\Wp$ and $\Ws$ for $W$ not of type $\I_2(p)$.

\begin{lemma}\label{lem:orient_is_abs_irr}
Let $G<\SOr_n$ be an irreducible rotation group that is normalized by a reflection $s$. If $n>2$ then the group $G$ is absolutely irreducible.
\end{lemma}
\begin{proof} Suppose that $G$ is not absolutely irreducible. The group $\Gt=\left\langle G,s\right\rangle$ is absolutely irreducible since it contains a reflection. Hence, the reflection $s$ permutes the irreducible components of the complexification of $G$. This implies $n=2$ and thus the claim follows.
  \end{proof}

\subsection{Rotation groups in low dimensions}
\label{sub:ps_low_dim}

All elements of $\SOr_2$ and $\SOr_3$ are rotations and thus every finite subgroup of $\SOr_2$ and $\SOr_3$ is a rotation group. This is not true for $\SOr_4$, but its finite subgroups and the rotation groups among them can still be described explicitly. We sketch this description here, a more detailed discussion can be found in \cite{MR0169108}. There are two-to-one covering maps of Lie groups $\varphi: \SU_2 \times \SU_2 \To \SOr_4$ and $\psi:\SU_2 \To \SOr_3$. Therefore, the finite subgroups of $\SOr_4$ can be determined based on the knowledge of the finite subgroups of $\SOr_3$. These are cyclic groups $C_n$ of order $n$, dihedral groups $\D_n$ of order $2n$ and the symmetry groups of a tetrahedron, an octahedron and an icosahedron, which are isomorphic to $\alt_4$, $\sym_4$ and $\alt_5$, respectively. Using the covering map $\psi$ one finds that the finite subgroups of $S^3$ are cyclic groups $\mathbf{C}_n$ of order $n$, binary dihedral groups $\mathbf{D}_n$ of order $4n$ and binary tetrahedral, octahedral and icosahedral groups denoted by $\mathbf{T}$, $\mathbf{O}$ and $\mathbf{I}$, respectively, and we set $\mathbf{V}=\mathbf{D}_2$. Except for $\mathbf{C}_n$ with odd $n$, these are two-to-one preimages of respective subgroups of $\SOr_3$, i.e. subgroups of $\SU_2$ of the form $\mathbf{C}_n$ with odd $n$ are the only ones that do not contain the kernel of $\psi$. In the following we identify $\SU_2$ with the unit quaternions in $\mathbb{H}$. Then the homomorphism $\varphi$ is explicitly given by
\[
		 \begin{array}{cccl}
		 \varphi : &\SU_2 \times \SU_2 				& \To  	&	\SOr_4 \\
		           		& (l,r) 	& \MTo  & \varphi((l,r)): q \MTo lqr^{-1}
		 \end{array}
\]
where $\R^4$ is identified with the algebra of quaternions $\mathbb{H}$ and has kernel $\{\pm(1,1)\}$. The classification result reads as follows \cite[p.~ 54]{MR0169108}.
\begin{proposition}
\label{prp:classification_of_finite_subgroups_of_so(4)}
For every finite subgroup $G \subgr  \SOr_4$ there are finite subgroups $\mathbf{L},\mathbf{R} \subgr  \SU_2$ with $-1\in \mathbf{L},\mathbf{R}$ and normal subgroups $\mathbf{L}_K \triangleleft \mathbf{L}$ and $\mathbf{R}_K \triangleleft \mathbf{R}$ such that $\mathbf{L}/\mathbf{L}_K$ and $\mathbf{R}/\mathbf{R}_K$ are isomorphic via an isomorphism $\phi: \mathbf{L}/\mathbf{L}_K \To \mathbf{R}/\mathbf{R}_K$ for which
\[
			G = \varphi(\{(l,r)\in \mathbf{L} \times \mathbf{R} | \phi(\pi_L(l))=\pi_R(r)\})
\]
holds, where $\pi_L: \mathbf{L} \To \mathbf{L}/\mathbf{L}_K$ and $\pi_R: \mathbf{R} \To \mathbf{R}/\mathbf{R}_K$ are the natural projections. In this case we write $G=(\mathbf{L}/\mathbf{L}_K;\mathbf{R}/\mathbf{R}_K)_{\phi}$. Conversely, a set of data $(\mathbf{L}/\mathbf{L}_K;\mathbf{R}/\mathbf{R}_K)_{\phi}$ with the above properties defines a finite subgroup $G$ of $\SOr_4$ by the equation above.
\end{proposition}
Given a finite subgroup $G \subgr  \SOr_4$, for $\mathbf{L} = \pi_1(\varphi^{-1}(G))$, $\mathbf{R} = \pi_2(\varphi^{-1}(G))$, $\mathbf{L}_K=\{l\in \mathbf{L} | \varphi((l,1))\in G\}$ and $\mathbf{R}_K=\{r\in \mathbf{R} | \varphi((1,r))\in G\}$ the quotient groups $\mathbf{L}/\mathbf{L}_K$ and $\mathbf{R}/\mathbf{R}_K$ are isomorphic and with the isomorphism $\phi$ induced by the relation $\varphi^{-1}(G)\subgr \mathbf{L} \times \mathbf{R}$ we have $G=(\mathbf{L}/\mathbf{L}_K;\mathbf{R}/\mathbf{R}_K)_{\phi}$. In most cases the conjugacy class of $(\mathbf{L}/\mathbf{L}_K;\mathbf{R}/\mathbf{R}_K)_{\phi}$ in $\SOr_4$ does not depend on the specific isomorphism $\phi$. However, there are a few exceptions. Since the finite subgroups of $\SU_2$ are invariant under conjugation \cite[p.~53]{MR0169108}, the groups $(\mathbf{L}/\mathbf{L}_K;\mathbf{R}/\mathbf{R}_K)_{\phi}$ and $(\mathbf{R}/\mathbf{R}_K;\mathbf{L}/\mathbf{L}_K)_{\phi^{-1}}$ are conjugate in $\Or_4$. For a list of finite subgroups of $\SOr_4$ we refer to \cite[p.~57]{MR0169108}. 

Elements of $\SOr_4$ of the form $\varphi(l,1)$ and $\varphi(1,r)$ for $l,r \in \SU_2$ are called \emph{left}- and \emph{rightscrews}, respectively. They commute mutually and for $l,r \in \SU_2$ there exist $a,b \in \SU_2$ and $\alpha,\beta \in \R$ such that $a^{-1}la=\cos(\alpha)+\sin(\alpha)i$ and $brb^{-1}=\cos(\beta)+\sin(\beta)i$. Then, with respect to the basis $\mathbb{B}=\{ab,aib,ajb,akb\}$, we have
\[
	\begin{array}{ccl}
		\varphi(l,1)_{\mathbb{B}}  = \left(
		  \begin{array}{cc}
		    R(\alpha) & 0 \\
		    0 & R(\alpha) \\
		  \end{array}
		\right)
  \end{array},
  \begin{array}{ccl}
		\varphi(1,r)_{\mathbb{B}}  = \left(
		  \begin{array}{cc}
		    R(\beta) & 0 \\
		    0 & R(-\beta) \\
		  \end{array}
		\right)
  \end{array}	
\]
and thus
\[
	\begin{array}{ccl}
		\varphi(l,r)_{\mathbb{B}}  = \left(
		  \begin{array}{cc}
		    R(\alpha+\beta) & 0 \\
		    0 & R(\alpha-\beta) \\
		  \end{array}
		\right).
  \end{array}
\]
where $R(\alpha)$ is a rotation about the angle $\alpha$. Consequently, $\varphi(l,r)$ is a rotation if and only if $\mathrm{Re}(l)=\mathrm{Re}(r)\notin\{\pm 1\}$. Using this observation it is possible to classify rotation groups in dimension $4$. The primitive rotation groups among them are singled out in \cite{Maerchik}. The groups of the form  $(\mathbf{C}_{km}/\mathbf{C}_m;\mathbf{R}/\mathbf{R}_K)$ listed in this paper under number 7.,...,11. preserve a complex structure and correspond to the primitive unitary reflection groups in dimension $2$ (cf. Section \ref{sub:unitary}). The groups in the list that come from real reflection groups are (cf. Section \ref{sub:reflection-groups}): $\Wp(\A_4)=(\mathbf{I}/\mathbf{C}_1;\mathbf{I}/\mathbf{C}_1)^*$, $\Ws(\A_4)=(\mathbf{I}/\mathbf{C}_2;\mathbf{I}/\mathbf{C}_2)^*$, $\Ws(\D_4)=(\mathbf{T}/\mathbf{T};\mathbf{T}/\mathbf{T})$, $\Wp(\F_4)=(\mathbf{O}/\mathbf{T};\mathbf{O}/\mathbf{T})$, $\Ws(\F_4)=(\mathbf{O}/\mathbf{O};\mathbf{O}/\mathbf{O})$ and $\Wp(\Hn_4)=(\mathbf{I}/\mathbf{I};\mathbf{I}/\mathbf{I})$. Here, the star $*$ indicates the choice of an outer automorphism (cf. Proposition \ref{prp:classification_of_finite_subgroups_of_so(4)}). The remaining primitive rotation groups appearing in \cite{Maerchik} are listed in Table \ref{tab:rot4}.

\begin{table}[ht] \centering
\begin{tabular}{ r | l | l }                      
  & rotation group & order \\
  \hline 
 1.&$(\mathbf{D}_{3m}/\mathbf{D}_{3m};\mathbf{T}/\mathbf{T})$  & $144m$  \\
 2.&$(\mathbf{D}_{m}/\mathbf{D}_{m};\mathbf{O}/\mathbf{O})$  & $96m$  \\
 3.&$(\mathbf{D}_{m}/\mathbf{C}_{2m};\mathbf{O}/\mathbf{T})$  & $48m$  \\
 4.&$(\mathbf{D}_{2m}/\mathbf{D}_{m};\mathbf{O}/\mathbf{T})$  & $96m$  \\
 5.&$(\mathbf{D}_{3m}/\mathbf{C}_{2m};\mathbf{O}/\mathbf{V})$  & $48m$  \\
 6.&$(\mathbf{D}_{m}/\mathbf{D}_{m};\mathbf{I}/\mathbf{I})$  & $240m$  \\
 7.&$(\mathbf{T}/\mathbf{T};\mathbf{O}/\mathbf{O})$  & $576$  \\
 8.&$(\mathbf{T}/\mathbf{T};\mathbf{I}/\mathbf{I})$  & $1440$  \\
 9.&$(\mathbf{O}/\mathbf{O};\mathbf{I}/\mathbf{I})$  & $2880$  \\
 \end{tabular}
\caption{Primitive rotation groups in $\Or_4$ that do not preserve a complex structure and are different from rotation groups of type $\Wp$ and $W^*$.}
\label{tab:rot4} 
\end{table}
\subsection{Monomial reflection-rotation groups}
\label{sub:monomial_examples}

An imprimitive linear group is called \emph{monomial} if it admits a system of imprimitivity consisting of one-dimensional subspaces (cf. Introduction). Examples for monomial irreducible reflection-rotation groups are the reflection group of type $\BC_n$ and its orientation preserving subgroup. To construct other examples let $H\subgr \sym_n$ be a permutation group generated by a set of double transpositions and $3$-cycles, e.g. (cf. \cite[p.~104]{Mikhailova})

\begin{enumerate}
  \item $H=H_5=\left\langle (1,2)(3,4),(2,3)(4,5)\right\rangle \subgr  \sym_5$, $H_5\cong \dih_5$
	\item $H=H_6=\left\langle (1,2)(3,4),(1,5)(2,3),(1,6)(2,4)\right\rangle \subgr  \sym_6$, $H_6\cong \alt_5$
	\item $H=H_7=\left\langle g_1,g_2,g_3 \right\rangle \subgr  \sym_7$, $H_7\cong \PSL_2(7)\cong \SL_3(2)$,
	\item $H=H_8=\left\langle g_1,g_2,g_3,g_4\right\rangle \subgr \sym_8$, $H_8\cong \AG_3(2)\cong\Z_2^3 \rtimes \SL_3(2)$
	\item $H=\alt_n \subgr  \sym_n $
\end{enumerate}
\noindent where
\[
	g_1=(1,2)(3,4),\; g_2=(1,5)(2,6),\; g_3=(1,3)(5,7),\; g_4=(1,2)(7,8).
\]
Regarding such a permutation group $H\subgr \sym_n$ as a subgroup of $\SOr_n$ yields a monomial rotation group, which is however not irreducible. Other examples of monomial reducible reflection-rotation groups are the diagonal subgroup $D(n)=D(W(\BC_n))$ of a reflection group of type $\BC_n$ and its orientation preserving subgroup $\Dp(n)=D(W(\D_n))$. Both groups are normalized by $\sym_n\subgr \SOr_n$. Therefore, we obtain a class of examples defined as semidirect products of $D(n)$ and $\Dp(n)$, respectively, with a permutation group $H\subgr \sym_n$ as above. We define $M_n=\Dp(n) \rtimes H_n$ for $n=5,\ldots,8$, $\Mt_n=D(n) \rtimes H_n$ for $n=5,\ldots,8$ and $\Mt(\D_n)=D(n) \rtimes \alt_n$. Moreover, we can define the following two exceptional examples of monomial irreducible rotation groups (cf. \cite[p.~104]{Mikhailova})
\[ 
	M_7^p=\left\langle g_1,g_2,g_3,g_5\right\rangle\subgr \SOr_7,\
	M^p_8=\left\langle g_1,g_2,g_3,g_4,g_5 \right\rangle\subgr \SOr_8,
\]
with $g_5 = (1,\overline{2})(3,\overline{4})$ where we write $(i,\overline{j})$ for the linear transformation that maps the basis vectors $e_i$ to $-e_j$ and $-e_j$ to $e_i$.

We record the following fact that can be checked by a computation. The groups $\AG_3(2)$ and $M_7^p$ are isomorphic and the restriction of the permutation representation of $\AG_3(2)$ described in $(iv)$ to $\R^7$ is equivalent to the natural representations of $M^p_7$ on $\R^7$.

\subsection{Nonmonomial imprimitive reflection-rotation groups}
\label{sub:imprimitive_example}
The imprimitive unitary reflection groups $G(m,p,n)$ give rise to a family of imprimitive rotation groups (cf. Section \ref{sub:unitary}). Related families of reflection-rotation groups can be constructed as follows. For a positive integer $m$ and $k=1,2$ the groups $\Wp(\I_2(m))$ and $W(\I_2(m))$ are normal subgroups of $W(\I_2(km))$ with abelian quotient. Hence, 
\[
		\As(km,k,n) = \left\{(g_1,\ldots,g_n)\in W(\I_2(km))^n | (g_1\cdots g_n) \in \Wp(\I_2(m)) \right\}
\]
and
\[
		\At(km,k,n) = \left\{(g_1,\ldots,g_n)\in W(\I_2(km))^n | (g_1\cdots g_n) \in W(\I_2(m)) \right\}
\]
are groups. We define
\[
		\Gs(km,k,n)=\As(km,k,n) \rtimes \sym_n< \SOr_{2n}
\]
and
\[
	 \Gt(km,k,l)=\At(km,k,l) \rtimes \sym_n <\Or_{2n}
\]
where the symmetric group $\sym_n$ permutes the components of $\As(km,k,n)$ and $\At(km,k,n)$, respectively. Let $s$, $r$ be the transformation of $\C^n$ defined by 
\[
	s(z_1,\ldots, z_n)=(\overline{z}_1,z_2,\ldots, z_n), \  r(z_1,\ldots, z_l)=(\overline{z}_1,\overline{z}_2,z_3,\ldots, z_l)
\]
Then we have
\[
		\Gs(km,k,n)=\left\langle G(km,k,n), r  \right\rangle, \ \Gt(km,k,n)=\left\langle G(km,k,n), s  \right\rangle
\]
where the complex groups on the right hand sides are regarded as real groups. In particular, the group $\Gs(km,k,n)$ is an imprimitive irreducible rotation group and the group $\Gt(km,k,n)$ is an imprimitive irreducible reflection-rotation group for $km\geq 3$ and $k=1,2$.

In dimension four, other examples can be constructed in the following way. Let $m$ and $k$ be positive integers and let $\varphi:\dih_k\To\dih_k$ be an involutive automorphism of the dihedral group of order $2k$ that maps reflections onto reflections. The data $\{(W(\I_2(km)),\Wp(\I_2(m)),W(\I_2(km))\}_{i\in \{1,2\}}$ together with this automorphism defines a reducible rotation group $D$ (cf. Theorem \ref{thm:clas_red_prg}). Since $\varphi$ has order $2$, the rotation that interchanges the two irreducible componets of $D$ normalizes $D$. We denote the rotation group generated by $D$ and this normalizing rotation by $\Gs(km,k,2)_{\varphi}$.

\subsection{Reducible reflection-rotation group}
\label{sub:reducible_examples}
We say that a reflection-rotation group $G$ is \emph{indecomposable} if it cannot be written as a product of subgroups that act in orthogonal spaces. Every reflection-rotation group splits as a product of indecomposable components. Basic examples for reducible but indecomposable rotation groups are $\Wp(\A_1\times \dots \times \A_1)$ and the diagonal subgroup $\Delta(W\times W)$ of the product of two copies of an irreducible reflection group $W<\Or_n$. The second example preserves a complex structure and coincides with the unitary reflection group of type $W$ considered as a real group. More generally, for an automorphism $\varphi: W \To W$ that maps reflections onto reflections the group
\[
		\Delta_\varphi(W\times W)=\{(g,\varphi(g)) \in W\times W | g \in W\}< \SOr_{2n}
\]
is a rotation group. The groups $\Delta_\varphi(W\times W)$ and $\Delta(W\times W)$ are conjugate in $\SOr_{2n}$, if and only if the automorphism $\varphi$ is realizable through conjugation by an element in $\Or_n$. This is possible if all labels of the Coxeter diagram of $W$ lie in $\{2,3,4,6\}$ \cite[Cor. 19, p.~7]{MR1997410}. However, reflection groups of type $\I_2(p)$, $\Hn_3$ and $\Hn_4$ admit automorphisms that map reflections onto reflections but cannot be realized through conjugation in $\Or_n$ \cite[pp.~31-32]{Franszen}. The exceptional rotation groups arising in this way for $W$ of type $\Hn_3$ and $\Hn_4$ do not preserve a complex structure (cf. Section \ref{sub:unitary}) and have not been studied in the context of \cite{Mikhailova} (cf. \cite[Thm.~2.1, p.~105]{Mikhailova}). General reducible but indecomposable reflection-rotation groups are extensions of the examples from this section by irreducible rotation groups we have described so far (cf. Section \ref{sec:class_reducible}).

\subsection{Exceptional primitive rotation groups}
\label{sub:primitive}
We have already seen a couple of primitive absolutely irreducible rotation groups. The rotation groups $\Wp(\A_n)$, $\Wp(\E_6)$, $\Wp(\E_7)$, $\Wp(\E_8)$, $\Ws(\A_5)$ and $\Ws(\E_6)$ belong to this class. In this section we describe two other examples.

\begin{lemma}\label{lem:exotic_psg_A5}
There exists a primitive absolutely irreducible rotation group isomorphic to the alternating group $\alt_5$. We denote it as $R_5(\alt_5) \subgr  \SOr_5$.
\end{lemma}
\begin{proof} We obtain a faithful linear representation of $\alt_5$ on $\R^5$ by restricting the nontrivial part of the permutation representation of $\sym_6$ to the image of an exceptional embedding $i: \alt_5 \subgr \sym_5 \To \sym_6$. This is the unique absolutely irreducible representation of $\alt_5$ in dimension 5 \cite[p.~2]{MR827219}. Since $i$ maps double transpositions to double transpositions the corresponding linear group $R_5(\alt_5) \subgr  \SOr_5$ is a rotation group isomorphic to $\alt_5$. The fact that $\alt_5$ is a simple group in combination with the results from Section \ref{sub:imprimitive} implies that $R_5(\alt_5)$ is primitive.
  \end{proof}

\begin{lemma}\label{lem:exotic_psg_psl}
There exists a primitive absolutely irreducible rotation group isomorphic to $\PSL_2(7)$. We denote it as $R_6(\PSL_2(7)) \subgr  \SOr_6$. The group $G=\left\langle R_6(\PSL_2(7)), \mathrm{-id}\right\rangle$ is not a rotation group.
\end{lemma}
\begin{proof} The group $\PSL_2(7)$ has a unique faithful and absolutely irreducible representation in dimension 6 \cite[p.~3]{MR827219} and we denote its image by $R_6(\PSL_2(7)) \subgr  \SOr_6$. It can be obtained by restricting the natural representation of $\sym_7$ on $\R^6$ to a subgroup described in Section \ref{sub:monomial_examples}, $(iii)$. This shows that $R_6(\PSL_2(7))$ is generated by rotations. The fact that $\PSL_2(7)$ is a simple group in combination with the results from Section \ref{sub:imprimitive} implies that $R_6(\PSL_2(7))$ is primitive.

Since the eigenvalues of a cycle $(1,\ldots,k)$ regarded as a linear transformation are the $k$-th roots of unity, the maximal dimension of the $-1$-eigenspace of a permutation $\sigma \in \sym_7$ acting on $\R^6$ is 3. This shows that all rotations contained in $G$ are also contained in $R_6(\PSL_2(7))$ and hence $G$ is not a rotation group.
  \end{proof}

\subsection{A new primitive rotation group}
\label{sec:group_L_sec}
The group $W(\I_2(4)) \subgr  \Or_2$ is a natural realization of the dihedral group $\dih_4$ of order $8$. Let $H$ be the tensor product of $3$ copies of $W(\I_2(4))$, i.e. $H=W(\I_2(4))\otimes W(\I_2(4)) \otimes W(\I_2(4)) \subgr  \SOr_8$. Then $H$ is an absolutely irreducible group of order $2^7$ and its normalizer $N=N_{\SOr_8}(H)$ contains rotations of order $2$, e.g. the linear transformations that interchange two $W(\I_2(4))$ factors. We would like to classify primitive rotation groups $G$ with $H<G<N$, as this problem occurs in our classification of rotation groups (cf. Proposition \ref{prp:primitive_dimension6_conlusion}). 

The images $A$ and $N(A)$ in $\mathrm{Int}(\mathfrak{so}_8)$ of the groups $H$ and $N$ are members of a series of finite subgroups of $\mathrm{Int}(\mathfrak{so}_{2^m})$ studied in connection with gradings of simple Lie algebras. Namely let $H_m\subgr \SOr_{2^m}$ be the tensor product of $m$ copies of $W(\I_2(4))$, i.e. $H=H_3$. Then the group $A_m=H_m/\{\pm 1\}$ is a so-called \emph{Jordan subgroup} of $\mathrm{Int}(\mathfrak{so}_{2^m})$. It is a $2$-elementary abelian group of order $2^{2m}$ and can be considered as a $2m$-dimensional vector space over $\mathbb{F}_2$. It is known that the assignment $Q(x)=0$ or $Q(x)=1$ for $x=\{\pm h\}$ depending on whether $h^2=1$ or $h^2=-1$ defines a nondegenerate quadratic form of Witt index $m$ and that the natural action of $N(H_m)$ on $A_m$ defines an isomorphism of the group $N(H_m)/H_m$ onto the orthogonal group $O(Q)$ \cite[Sect. 3.12, Example 4, p.~126]{MR1349140},\cite{MR0379748,MR1191170}. A quadratic form on a vector space over $\mathbb{F}_2$ is called nondegenerate, if the bilinear form $f(x,y)=Q(x+y)+Q(x)+Q(y)$ is nondegenerate. Its Witt index is defined to be the maximal dimension of a singular subspace, i.e. a subspace on which the quadratic form vanishes identically (cf. \cite[Sect.~I.16, p.~34]{MR0158011}). Let $O^+(Q)$ be the index $2$ subgroup of $O(Q)$ whose elements have Dickson invariant $0$ (cf. \cite[Sect.~II.10, p.~65]{MR0158011}) and let $L$ be the preimage of $O^+(Q)$ in $N$. We are going to show that the group $L$ is a primitive rotation group and that it is the only such group $G$ with $H<G<N$ (cf. Lemma \ref{lem:pssgroups_L}).

For a bivector in $\bigwedge^2 \mathbb{F}^4_2$ we divide its exterior square computed over $\Z$ by 2 and consider the result modulo $2$. This assignment defines a quadratic form $Q'$ with the wedge product $\bigwedge : \bigwedge^2 \mathbb{F}^4_2 \times \bigwedge^2 \mathbb{F}^4_2 \To \mathbb{F}_2$ as associated bilinear form. Hence, the form $Q'$ is nondegenerate and has Witt index $3$, a three-dimensional singular subspace being $U\wedge \mathbb{F}^4_2$ for a one-dimensional subspace $U\subset \mathbb{F}^4_2$. It follows that there exists an isomorphism between $(H/\{\pm 1\},Q)$ and $(\bigwedge^2 \mathbb{F}^4_2,Q')$ \cite[Sect.~I.16, p.~34]{MR0158011}. We identify $H/\{\pm 1\}$ and $\bigwedge^2 \mathbb{F}^4$ via such an isomorphism. In this way we obtain an embedding of $\SL_4(2)$ into $O(Q)$. Since both $O^+(Q)$ and the image of $\SL_4(2)$ are normal subgroups of $O(Q)$ isomorphic to the simple group $A_8$ \cite[p.~22]{MR827219} we can identify $\SL_4(2)$ with $O^+(Q)$.
 
Due to a Witt type theorem in characteristic $2$ proved by C. Arf every isomorphism between singular subspaces of $H/\{\pm 1\}$ can be extended to an isometry of $H/\{\pm 1\}$ \cite[Sect.~I.16, p.~36; cf. Sect.~I.11, p.~21]{MR0158011}. In particular, flags of singular subspaces of $H/\{\pm 1\}$ with the same signature are $O(Q)$-equivalent. Maximal singular subspaces of $H/\{\pm 1\}$ of the form $U\wedge \mathbb{F}^4_2$ for a one-dimensional subspace $U\subset \mathbb{F}^4_2$ are not $O^+(Q)$-equivalent to maximal singular subspaces of the form $U\wedge U$ for a three-dimensional subspace $U\subset \mathbb{F}^4_2$. For, a subspace of the first type can be annihilate by wedging with a one-dimensional subspace of $\mathbb{F}^4_2$ whereas a subspace of the second type cannot. Hence, the set of maximal singular subspaces of $H/\{\pm 1\}$ decomposes into two $O^+(Q)$-orbits represented by these two types of subspaces. Each orbit conjoint with a trivial element inherits a vector space structure over $\mathbb{F}_2$ from $\mathbb{F}^4_2$. Three different maximal singular subspaces belong to a two-dimensional subspace of this vector space, if and only if they intersect in a one-dimensional singular subspace. Hence, $H/\{\pm 1\}$ contains $30$ maximal singular subspaces. Since, every maximal singular subspace contains $7$ one-dimensional (singular) subspaces and every one-dimensional singular subspace is contained in $3$ maximal singular subspaces from each $O^+(Q)$-orbit, we see that $H/\{\pm 1\}$ contains $35$ one-dimensional singular subspaces. Moreover, every two-dimensional singular subspace is contained in precisely one maximal singular subspace from each $O^+(Q)$-orbit.

A representative of a maximal singular subspace of $H/\{\pm 1\}$ is given by
\[
(W(\I_2(2))\otimes W(\I_2(2)) \otimes W(\I_2(2)))/\{\pm 1\}
\]
where $W(\I_2(2))<W(\I_2(4))$ is a Klein four-group. The preimage of an $i$-dimensional singular subspace of $H/\{\pm 1\}$ in $H$ is an abelian normal subgroup of $H$ of order $2^{i+1}$ with respect to which the space $\R^8$ decomposes into an orthogonal sum of $2^{3-i}$-dimensional weight spaces. Denote the collection of $2^{3-i}$-dimensional subspaces obtained in this way from the $i$-dimensional singular subspaces of $H/\{\pm 1\}$ by $K_{2^{3-i}}$ and the corresponding collections of involutions whose $-1$-eigenspaces are the subspaces from $K_{2^{3-i}}$ by $\mathfrak{R}_{2^{3-i}}$. We record the following lemma.

\begin{lemma} \label{lem:H-orbits} The group $N$ acts transitively on $K_{2^{3-i}}$, $i=1,2,3$. The $H$-orbits in $K_{2^{3-i}}$ have order $2^i$ and are in one-to-one correspondence with the $i$-dimensional singular subspaces of $H/\{\pm 1\}$. In particular, every element of $K_{2^{3-i}}$ uniquely determines an $i$-dimensional singular subspace of $H/\{\pm 1\}$.
\end{lemma}
\begin{proof} Since the preimage of a singular subspace of $H/\{\pm 1\}$ in $H$ is a normal subgroup, its weight spaces are permuted by $H$. The group $H$ being irreducible implies that these weight spaces are transitively permuted by $H$. Now the claim follows since the group $N$ acts transitively on singular subspaces of $H/\{\pm 1\}$ of the same dimension.
  \end{proof}

A  linear transformation $f$ of a vector space $V$ is called a $\emph{transvection}$, if there exist $e \in V\backslash \{0\}$ and a nontrivial linear form $\alpha$ on $V$ with $\alpha(e)=0$ such that $f(v)=v+\alpha(v)e$ for all $v\in V$.

\begin{lemma}\label{lem:rotation normalize}
The rotations in $\mathfrak{R}_2$ belong to $L$ and project to transvections in $\SL_2(4)$.
\end{lemma}
\begin{proof} It is straightforward to check that some rotation in $\mathfrak{R}_2$ normalizes the group $H$. Since the normalizer of $H$ acts transitively on $\mathfrak{R}_2$ all of them do. Let $r \in \mathfrak{R}_2$ be a rotation. We identify $V=\mathbb{F}^4_2$ with the set of maximal singular subspaces of  $H/\{\pm 1\}$ that are $O^+(Q)$-equivalent to  $U\wedge \mathbb{F}^4_2$, $U\subset \mathbb{F}^4_2$ being one-dimensional, conjoint with a trivial element. Let $W_2$ be the two-dimensional singular subspace of $H/\{\pm 1\}$ determined by $r$ (cf. Lemma \ref{lem:H-orbits}) and let $W_3$ be the unique maximal singular subspace of $H/\{\pm 1\}$ with $W_2\subset W_3 \in V$. The union $U= \bigcup_{\sigma\in W_2\backslash \{0\}} \sigma \subset V$, where the nontrivial elements $\sigma\in W_2$ are regarded as two-dimensional subspaces of $V$ via the identification $H/\{\pm 1\}=\bigwedge^2 V$, is a three-dimensional subspace of $V$. Indeed, for $v_1\in \sigma_1\in W_2\backslash \{0\}$ and $v_2\in \sigma_2\in W_2\backslash \{0\}$ with $v_1 \wedge v_2 \notin W_2$ we have $(v_1+v_2)\wedge (\sigma_1+\sigma_2)=v_2\wedge \sigma_1+v_1\wedge \sigma_2=0$ because of $\sigma_1\wedge \sigma_2=0$ and thus $v_1+v_2 \in \sigma_1+\sigma_2 \subset U$. For a maximal singular subspace $W \in U$ there exists some nontrivial $h\in H$ with $\{\pm h\} \in W \cap W_2$ (for $W\neq W_3$ take $\{ \pm h\}=W\wedge W_3$) and we can assume that $\mathrm{Fix}(h)\subset  \mathrm{Fix}(r)$. Therefore $r$ leaves some weight space corresponding to $W$ invariant. Since $H$ acts transitively on the weight spaces corresponding to $W$, the rotation $r$ permutes them and hence fixes $W$. This means that $r$ acts trivially on $U$. In particular, it leaves $V$ invariant and thus belongs to $L$. Let $W' \in V \backslash U$. If $r$ would be the identity on $V$ and thus on $H/\{\pm 1\}=\bigwedge^2 V$, we had $r \in H$ \cite[Thm. 3.19, (1), p.~126]{MR1349140}, a contradiction. Hence, $\overline{W}=W'+r(W') \in U$ is nontrivial and for any $W'' \in V \backslash U$ we have $r(W'')=r(W''+W')+r(W')=W''+\overline{W}$, because of $W''+W'\in U$. Consequently, the rotation $r$ acts on $V$ like the transvection defined by $e=\overline{W} \in V$ and the linear form $\alpha$ corresponding to the three-dimensional subspace $U\subset V$. 
  \end{proof}
Moreover, we have
\begin{lemma}\label{lem:is_isomorphism}
The group $L$ is a rotation group generated by $\mathfrak{R}_{2}$. The set $\mathfrak{R}_{2}$ has order $420$.
\end{lemma}
\begin{proof} Since all transvections in $\SL_4(2)$ are conjugate and generate $\SL_4(2)$ (cf. \cite[Sect.~II.1, p.~37]{MR0158011}) we see that $\left\langle \mathfrak{R}_{2}\right\rangle$ maps onto $\SL_4(2)$ by the preceding lemma. Moreover, since every one-dimensional singular subspace of $H/\{\pm 1\}$ is contained in a two-dimensional singular subspace of $H/\{\pm 1\}$, every element $h\in H$ with $Q(\{\pm h\})=0$ can be written as a product of rotations in $\mathfrak{R}_{2}$. This implies $H \subset \left\langle \mathfrak{R}_{2}\right\rangle$ and thus $L$ is generated by $\mathfrak{R}_{2}$. The group $\SL_4(2)$ contains $(2^4-1)(2^3-1)=105$ transvections and the $H$-orbit of any rotation in $\mathfrak{R}_2$ contains $4$ rotations by Lemma \ref{lem:H-orbits}. We conclude that the set $\mathfrak{R}_2$ has order $420$.
  \end{proof}

Now we can show the following statement.

\begin{proposition}\label{prp:exotic_psg_L}
The rotation group $L\subgr \SOr_8$ is primitive and absolutely irreducible.
\end{proposition}
\begin{proof} The group $L$ is absolutely irreducible, as it contains the absolutely irreducible group $H$. The fact that $L/H\cong \SL_2(4)$ is a simple group implies that all other normal subgroups of $L$ are contained in $H$. The nontrivial center of $L$ is the only subgroup of $H$ normalized by all rotations in $L$. Therefore, $H$ and $\{\pm \mathrm{id}\}$ are the only nontrivial normal subgroups of $L$. Assume that $L$ is an imprimitive group and let $\R^8=V_1 \oplus \dots \oplus V_k$, $k\in \{4,8\}$, be a hypothetical decomposition into subspaces that are permuted by $L$. The diagonal subgroup $D=D(L)$ with respect to this system of imprimitivity is normal in $L$ and satisfies $|D|\geq|L/8!|=64$, because $L/D$ embeds into the symmetric group $\sym_k$. Therefore, we would have $D=H$ which is impossible, since $H$ is irreducible and $D$ is not. Hence the claim follows.
  \end{proof}

Our next aim is to show that every rotation in $N$ is contained in $\mathfrak{R}_2$. To this end we first construct certain rotations in $\mathfrak{R}_2$ that are needed in the proof.

\begin{lemma}\label{lem:span_rotation} The span of any two distinct one-dimensional weight spaces corresponding to a maximal singular subspace of $H/\{\pm 1\}$ is contained in $K_2$, i.e. it occurs as a weight space of a two-dimensional singular subspace of $H/\{\pm 1\}$ and corresponds to a rotation in $\mathfrak{R}_2$.
\end{lemma}
\begin{proof} Let $H_0<H$ be the subgroup corresponding to a maximal singular subspace of $H/\{\pm 1\}$ and let $v_1,v_2\in \R^8$ be unit vectors spanning two different weight spaces of $H_0$. One only needs to find $h_1,h_2 \in H_0$ that project onto linearly independent elements in $H/\{\pm 1\}$ such that $v_1$ and $v_2$ belong to the same eigenspace of $h_i$, $i=1,2$. A case differentiation shows that this is always possible.
  \end{proof}

Let $\{e_1,e_2\}$ be the standard basis of $\R^2$ and let $\{\varepsilon_i|i\in\{1,\ldots,8\}\}=\{e_i\otimes e_j\otimes e_k|i,j,k\in\{1,2\}\}$ be the induced basis of $\R^8=\R^2\otimes \R^2 \otimes \R^2$ ordered lexicographically. We set $V_1:=\left\langle \varepsilon_1,\varepsilon_2,\varepsilon_3,\varepsilon_4 \right\rangle$,  $\sigma_1:=\left\langle \varepsilon_1,\varepsilon_2 \right\rangle$ and $\sigma_2:=\left\langle \varepsilon_3,\varepsilon_4 \right\rangle$. Clearly, we have $V_1 \in K_4$ and $\sigma_1,\sigma_2 \in K_2$. For a reflection $s\in N_{\Or_2}(W(\I_2(4))) \backslash W(\I_2(4))$ with $s(\sqrt{2} e_1)=e_1+e_2$ we set $\nu_2= \mathrm{id}\otimes s \otimes \mathrm{id} \in N$ and $\nu_3= \mathrm{id}\otimes  \mathrm{id} \otimes s \in N$. Then with
\[
		\alpha_1:=\nu_2(\varepsilon_1)=\varepsilon_1+\varepsilon_3,\ \alpha_2:=\nu_2(\varepsilon_2)=\varepsilon_2+\varepsilon_4,
\]
and
\[
\beta_1:=\nu_3(\varepsilon_1)=\varepsilon_1+\varepsilon_2,\ \beta_2:=\nu_3(\varepsilon_5)=\varepsilon_5+\varepsilon_6
\]
we have $\left\langle \alpha_1 \right\rangle, \left\langle \alpha_2 \right\rangle, \left\langle \beta_1 \right\rangle, \left\langle \beta_2 \right\rangle \in K_1$ because of $\nu_2,\nu_3 \in N$ and $\sigma=\left\langle \alpha_1,\alpha_2 \right\rangle, \tau_1=\left\langle \varepsilon_1,\varepsilon_5 \right\rangle, \tau_2=\left\langle \beta_1,\beta_2 \right\rangle  \in K_2$ by Lemma \ref{lem:span_rotation}. Let $r,r_1,r_2 \in \mathfrak{R}_2$ be the rotations corresponding to $\sigma$, $\tau_1$, and $\tau_2$, respectively, and let $R=\left\langle r_1r_2 \right\rangle <N$ be the group generated by $r_1r_2$. The group $R$ is isomorphic to a dihedral group of order $8$ and leaves $\sigma_1$ and $\sigma_2$ invariant. The rotation $r$ interchanges $\sigma_1$ and $\sigma_2$ and thus so do the conjugates of $r$ under the group $R$. Hence, there are at least $8$ different rotations in $\mathfrak{R}_2$ that interchange $\sigma_1$ and $\sigma_2$.

\begin{lemma}\label{lem:rot_contained_in_L} Every rotation in $N$ is contained in $\mathfrak{R}_2$.
\end{lemma}
\begin{proof} Let $g \in N$ be a rotation. Then there exists a one-dimensional singular subspace of $H/\{\pm 1\}$ spanned by some $\{\pm h\}$  with $g\{\pm h\}g^{-1}\neq \{\pm h\}$. For, otherwise $g$ would be the identity on $H/\{\pm 1\}$ and thus contained in $H$ \cite[Thm. 3.19, (1), p.~126]{MR1349140}. We set $h'=ghg^{-1}$. Since $g$ is a rotation the intersection $\mathrm{Fix}(h)\cap\mathrm{Fix}(h')$ is nontrivial. The eigenvalue structure of the elements in $H$ and the fact that $h\neq h'$ implies that $hh'$ has order $2$. Hence, the group $H_0=\left\langle h,h'\right\rangle$ projects onto a two-dimensional singular subspace of $H/\{\pm 1\}$ that is contained in a maximal singular subspace $W$. Two of the four weight spaces of $H_0$ are pointwise fixed by $g$, the other two are interchanged by $g$. In particular, the rotation $g$ has order $2$. The one-dimensional weight spaces defined by $W$ are contained in the two-dimensional weight spaces of $H_0$. Since $H$ acts transitively on the weight spaces corresponding to $W$ and since $g$ fixes one of them, it permutes the others. Due to the transitivity of the $N$ action on flags of singular subspaces of $H/\{\pm 1\}$ we can assume that $g\sigma_1=\sigma_2$ and that the weight spaces corresponding to $W$ are spanned by $\varepsilon_1,\ldots,\varepsilon_8$. There are only $8$ rotations in $\SOr_8$ with these properties and we have seen above that all of them are contained in $\mathfrak{R}_2$. Hence the claim follows.
  \end{proof}

We also need the following lemma. Recall that the action of $L$ on $H/\{\pm 1\}=\bigwedge^2 \mathbb{F}^4_2$ descends to an action on $\mathbb{F}^4_2$.
\begin{lemma}\label{lem:symplectic_form}
Let $G<L$ be a rotation group and suppose that $G$ leaves a symplectic form on $\mathbb{F}^4_2$ invariant. Then the group $H$ is not contained in $G$.
\end{lemma}
\begin{proof} Suppose $G<L$ is a rotation group that leaves a symplectic form $B$ on $\mathbb{F}^4_2$ invariant. The form $B$ defines a nontrivial $G$-invariant linear form $\beta$ on $\bigwedge^2 \mathbb{F}^4_2$. By duality with respect to the nondegenerate bilinear form $\bigwedge : \bigwedge^2 \mathbb{F}^4_2 \times \bigwedge^2 \mathbb{F}^4_2 \To \mathbb{F}_2$, the form $\beta$ in turn gives rise to a nontrivial $G$-invariant bivector $b\in \bigwedge^2 \mathbb{F}^4_2$ that corresponds to a $G$-invariant coset $\{\pm h\}\in H/\{\pm 1\}$ for some nontrivial element $h\in H$. The two (possibly complex) four-dimensional eigenspaces of $h$ corresponding to different eigenvalues cannot be permuted by a rotation. Hence, the group $G$, being generated by rotations, not only fixes $\{\pm h\}$ but also $h$. Since the bilinear form associated with $Q$ is nondegenerate, the center of $H$ is given by $\{\pm 1\}$ and thus the group $H$ cannot be completely contained in $G$.
  \end{proof}

Now we can prove
\begin{lemma}\label{lem:pssgroups_L}
The only primitive rotation group $G$ with $H<G\subgr  N$ is the group $L$.
\end{lemma}
\begin{proof} 
Let $G<N$ be a rotation group. By Lemma \ref{lem:rot_contained_in_L} we have $G<L$ and thus we can consider the action of $G$ on $\mathbb{F}^4_2$. If there exists a one- or three-dimensional $G$-invariant subspace $U$ of $\mathbb{F}^4_2$, then $G$ leaves a maximal singular subspace of $H/\{\pm 1\}=\bigwedge^2 \mathbb{F}^4_2$ invariant (either $U\wedge \mathbb{F}^4_2$ or $U\wedge U$ depending on whether $U$ is one- or three-dimensional). The corresponding collection of weight spaces defines a system of imprimitivity of $G$ and thus $G$ is imprimitive in this case. If there exists a two-dimensional invariant subspace of $\mathbb{F}^4_2$, then the group $G$ fixes a one-dimensional singular subspace of $H/\{\pm 1\}$ spanned by some $\{\pm h\}$. Again, since the group $G$ is generated by rotations it normalizes $h$ (cf. proof of Lemma \ref{lem:symplectic_form}) and is thus reducible. Otherwise, the group $G$ acts irreducibly on $\mathbb{F}^4_2$. Since its image in $\SL_4(2)$ is generated by transvections (cf. Lemma \ref{lem:rot_contained_in_L} and Lemma \ref{lem:is_isomorphism}), it either preserves a symplectic form on $\mathbb{F}^4_2$ or we have $G=L$ \cite[p.~108]{MR0237660}. In the first case the group $H$ is not contained in $G$ by Lemma \ref{lem:symplectic_form} and thus the claim follows. 
  \end{proof}

Finally, we explain how the rotation group $L$ is connected to a reflection group of type $\E_8$. Let $R_1$ and $R_2$ be given by the sets of vectors 
\[
\pm \varepsilon_i\pm \varepsilon_j \ (i<j), \ \frac{1}{2} \sum_{i=1}^8 \pm \varepsilon_i \ (\textrm{even number of + signs})
\]
and
\[
\pm \varepsilon_i, \ (\pm \varepsilon_i\pm \varepsilon_{i+1}\pm \varepsilon_j \pm\varepsilon_{j+1})/2, \ i\neq j, \ i,j \in \{1,3,5,7\},\]
\[  (\pm \varepsilon_i\pm \varepsilon_j \pm \varepsilon_k \pm\varepsilon_l)/2, \ i \in \{1,2\},\ j \in \{3,4\},\ k \in \{5,6\},\ l \in \{7,8\},
\]\[i+j+k+l \equiv 0 \textrm{ mod }2 \]
respectively. Then $R_1$ and $R_2$ are root systems of type $\E_8$ permuted by an involution of $N$. Moreover, a computation shows that the vectors in $R_1$ and $R_2$ span the subspaces in $K_1$ corresponding to the two orbits of $L$, so that a two-dimensional subspace of $\R^8$ belongs to $K_2$ if and only if its intersection with $R_1 \cup R_2$ is a root system of type $\I_2(4)$ and so that a four-dimensional subspace of $\R^8$ belongs to $K_4$ if and only if its intersection with $R_1 \cup R_2$ is a root system of type $\F_4$. In particular, we have $L=\left\langle \mathfrak{R}\right\rangle \subgr W(R_1)\cap W(R_2)$.

\begin{proposition}\label{prp:L_und_E8} The group $L$ is the intersection of two reflection groups of type $\E_8$ permuted by $N$. More precisely, $L=W(R_1)\cap W(R_2)$.
\end{proposition}
\begin{proof} Since $W(R_1)\cap W(R_2)$ leaves $K_4$ invariant, it normalizes $H$ and thus we have $L \subgr W(R_1)\cap W(R_2) < N$. An element $g \in N \backslash L$ satisfies $g(R_1)=R_2$ and is therefore not contained in $W(R_1)\cap W(R_2)$. Hence, the claim follows, as $L$ has index $2$ in $N$.
  \end{proof}

\subsection{Properties of exceptional rotation groups}\label{sub:excep_subgroups}
The monomial rotation group $M_8$ contains the group $H$ from the preceding section as a normal subgroup and is thus itself a subgroup of the rotation group $L$. In fact, every subgroup of $L$ which is maximal among subgroups fixing one of the systems of imprimitivity of $H$ described above is of this type. Therefore the group $L$ also contains the rotation groups $R_6(\PSL_2(7))$, $M_7$, $M^p_7$ and $M^p_8$ as subgroups (cf. Section \ref{sub:monomial_examples}). In this section we record some of their properties that are related to properties of the corresponding quotient spaces considered in \cite{Mikhailova}.

\begin{lemma} The rotation groups listed in Theorem \ref{thm:clas_red_psg}, $(v)$ only contain rotations of order $2$.
\end{lemma}
\begin{proof}
For the groups $R_5(\alt_5)$, $M_5$ and $M_6$ the claim can be readily checked. For the group $L$ it follows from Lemma \ref{lem:rot_contained_in_L} and thus it also holds for its subgroups $R_6(\PSL_2(7))$, $M_7$, $M_8$, $M^p_7$ and $M^p_8$.
  \end{proof}

We denote the plane systems defined by $M_5$, $M_6$, $M_7$, $M_8$, $M^p_7$, $M^p_8$, $R_5(\alt_5)$, $R_6(\PSL_2(7))$ and $L$ as $\Pp_5$, $\Pp_6$, $\Pp_7$, $\Pp_8$, $\Qq_7$, $\Qq_8$, $\Rr_5$, $\Ss_6$ and $\Tt_8$ and the corresponding rotation group by $M$, e.g. $L=M(\Tt_8)$.

\begin{lemma}\label{lem:isotropyL} All isotropy groups of the rotation groups $R_6(\PSL_2(7))$, $M_7^p$, $M_8^p$ and $L$ are rotation groups (cf. Section \ref{sec:isotropy}).
\end{lemma}
\begin{proof}
Let $G$ be one of the groups listed above. The claim follows if one can show that each element $g \in G$ can be written as a composition of rotations in $G$ whose fixed point subspace contains the fixed point subspace of $g$. It suffices to check this property for one representative in each conjugacy class of $G$. For the listed groups this can be verified with a computer algebra system like GAP.
  \end{proof}

\begin{lemma}
\label{lem:sub_psl}
The rotation group $M(\Ss_6)=R_6(\PSL_2(7))$ of order $168=2^{3}\cdot3\cdot 7$ contains a rotation group isomorphic to $\sym_4$.
\end{lemma} 
\begin{proof} The double transpositions $(1,7)(3,5)$, $(1,5)(3,7)$ and $(1,4)(6,7)$ generate a subgroup of the rotation group $R_6(\PSL_2(7))\subgr \sym_7\subgr \SOr_6$ (cf. Lemma \ref{lem:exotic_psg_psl}) isomorphic to $\sym_4$.
  \end{proof}

\begin{lemma}\label{lem:sub_psl_2(7)}
The rotation group $M(\Qq_7)=M^p_7$ of order $1344=2^{6}\cdot3\cdot 7$ contains rotation groups of order $192= 2^6\cdot 3$ and $168=2^{3}\cdot3\cdot 7$.
\end{lemma}
\begin{proof} The rotation group generated by $(1,\overline{3})(2,\overline{4})$, $(2,4)(5,7)$, $(2,3)(6,7)$ and $(3,\overline{4})(5,\overline{6})$ is a reducible subgroup of $M^p_7$ (cf. Section \ref{sub:monomial_examples}) of type
\[
\{(W(\D_4),D(W(\D_4)),W(\D_4),\Gamma(\A_{3})),(W(\A_3),\Wp(\A_1 \times \A_1 \times \A_1), W(\A_3),\circ-\circ)\}.
\]
Moreover, the rotation group $R_6(\PSL_2(7))$ of order $168$ is contained in $M^p_7$ (cf. Lemma \ref{lem:exotic_psg_psl} and Section \ref{sub:monomial_examples}).
  \end{proof}

\begin{lemma}
\label{lem:sub_GpAg}
The rotation group $M(\Qq_8)=M^p_8$ of order $2^{10}\cdot3\cdot 7$ contains a reducible rotation group $G$ of order $2^{9}\cdot3$ with $k=2$ and
\[
	(G_i,H_i,F_i,G_i/H_i)=(W(\D_4),D(W(\D_4)),W(\D_4),W(\A_3)),
\]
$i=1,2$ (cf. Theorem \ref{thm:reducible_pairs}), which is normalized by an element $h$ of order $2$ that interchanges the irreducible components of $G$. Moreover, it contains the rotation group $R_6(\PSL_2(7))$ of order $2^{3}\cdot3\cdot 7$.
\end{lemma}
\begin{proof} The rotations $(1,\overline{5})(4,\overline{8})$, $(1,6)(3,8)$, $(2,\overline{5})(3,\overline{8})$, $(3,7)(4,8)$ and $(3,4)(5,6)$ generate a subgroup $G$ of $M_8^p\subgr \sym_8$ (cf. Section \ref{sub:monomial_examples}) that leaves the subspace $\left\langle \varepsilon_1,\varepsilon_2,\varepsilon_5,\varepsilon_6 \right\rangle$ and its orthogonal complement invariant. In fact, it is a reducible rotation group of type $(G_i,H_i,F_i,G_i/H_i)=(W(\D_4),D(W(\D_4)),W(\D_4),W(\A_3))$, $i=1,2$. The involution $h=(1,8)(2,7)(3,6)(4,5)$ is contained in $M^p_8$, normalizes the group $G$ and interchanges its two irreducible subspaces. The rotation group $R_6(\PSL_2(7))$ is contained in $M^p_8$ as well (cf. Lemma \ref{lem:exotic_psg_psl} and Section \ref{sub:monomial_examples}).
  \end{proof}

\begin{lemma}
The rotation group $M(\Tt_8)=L$ of order $2^{13}\cdot 3^2 \cdot 5 \cdot 7$ contains the rotation group $M_8$ of order $2^{13}\cdot 3 \cdot 7$ and unitary reflection groups $W(\mathcal{F}_4)$ and $W(\mathcal{N}_4)$ of order $2^7\cdot 3^2$ and $2^9\cdot 3\cdot 5$, respectively.
\end{lemma}
\begin{proof}
It follows from the description of a \emph{line system} of type $\mathcal{O}_4$ (\cite[Sect. 6.2, p.~109]{MR2542964}, cf Section \ref{sub:unitary}) and the remark preceding Proposition \ref{prp:L_und_E8} that $L$ contains a unitary reflection group of type $W(\mathcal{O}_4)$, which itself contains unitary reflection groups of type $W(\mathcal{F}_4)$ and $W(\mathcal{N}_4)$ \cite[Sect. 6.2, p.~109]{MR2542964}.
  \end{proof}

\section{Irreducible rotation groups}
\label{sec:class_irreducible}
In this section we prove the classification of irreducible rotation groups, i.e. Theorem \ref{thm:clas_red_psg}. Let $G\subgr \SOr_n$ be an irreducible rotation group. If the complexification of $G$ is reducible, then $G$ is an irreducible unitary reflection group that is not the complexification of a real reflection group, considered as a real group by Lemma \ref{lem:abs_irr_or_complex} and Lemma \ref{lem:unitary_is_pseudo}. Hence, we are in case $(i)$ if $n=2$ or in case $(ii)$ or $(iii)$ if $n>2$ of Theorem \ref{thm:clas_red_psg}. Otherwise $G$ is absolutely irreducible. The classification of imprimitive absolutely irreducible rotation groups and primitive absolutely irreducible rotation groups in dimensions $n\geq 5$ is treated separately in the following two sections. Together with the classification in dimension $n\leq 4$ treated in Section \ref{sub:ps_low_dim}, the results of these sections form a complete proof of Theorem \ref{thm:clas_red_psg}.

\subsection{Imprimitive rotation groups}
\label{sub:imprimitive}
For a finite imprimitive group $G$ we can always assume that the subspaces $V_1,\ldots, V_l$ constituting a system of imprimitivity for $G$ are orthogonal and that $G\subgr \SOr_n$. If $G$ is moreover an irreducible rotation group, then it acts transitively on the set of these subspaces and thus all of them have the same dimension, either one or two. In the first case the group is called \emph{monomial}. The classification of absolutely irreducible monomial and nonmonomial imprimitive rotation groups is treated separately in the following two paragraphs.\\

\textbf{Monomial rotation groups.} Assume that $G$ is monomial. Since it is also orthogonal, each row and each column of any element of $G$ contains precisely one element from $\{\pm 1\}$. In particular, $G$ is contained in a reflection group of type $\BC_n$. Therefore, we obtain a homomorphism from $G$ to the symmetric group $\sym_n$ with the diagonal matrices $D$ of $G$ as kernel. Its image isomorphic to $G/D$ is a transitive subgroup of $\sym_n$ generated by transpositions, double transpositions and $3$-cycles. Such groups are classified in \cite[Thm. 2.1, p.~500]{MR570727}.
\begin{theorem}\label{thm:class_perm_group}Let $H$ be a transitive permutation subgroup of $\sym_n$ generated by a set of transpositions, double transpositions and $3$-cycles such that $H$ does not admit a two-dimensional system of imprimitivity, i.e. a partition of $\{1,\ldots,n\}$ into subsets of order two that are interchanged by $H$. Then, up to conjugation, $H$ is one of the following groups.
\begin{enumerate}
	\item $\sym_n$
	\item $\alt_n$
	\item $H_5=\left\langle (1,2)(3,4),(2,3)(4,5)\right\rangle \subgr  \sym_5$, $H_5\cong \dih_5$
	\item $H_6=\left\langle (1,2)(3,4),(2,3)(4,5),(3,4)(5,6)\right\rangle \subgr  \sym_6$, $H_6\cong \alt_5$
	\item $H_7=\left\langle (1,2)(3,4),(1,3)(5,6),(1,5)(2,7)\right\rangle \subgr  \sym_7$, $H_7\cong \PSL_2(7)\cong \SL_3(2)$
	\item $H_8=\left\langle H_7,(5,6)(7,8)\right\rangle\subgr  \sym_8$, $H_8\cong \AG_3(2)\cong \Z_2^3 \rtimes \SL_3(2)$.
\end{enumerate}
\end{theorem}
For each permutation group $H$ described above there exists a monomial rotation group $M\subgr \SOr_n$ whose diagonal subgroup $D=D(n)$ contains all linear transformations that change the sign of an even number of coordinates (cf. Section \ref{sub:monomial_examples}) such that $M/D \cong H$. Except in case $(i)$ this is a semidirect product of the permutation group $H$ with $D(n)$.

In \cite[Table I-III, p.~503]{MR570727} Huffman classifies irreducible monomial groups over the complex numbers that are generated by transformations with an eigenspace of codimension two. These tables contain all complexified monomial absolutely irreducible rotation groups. Together with \cite[p.~90]{MR608821} they imply the following result, where we write $(i,\overline{j})$ for the linear transformation that maps $e_i$ to $-e_j$, $-e_j$ to $e_i$ and all other standard basis vectors to itself.
\begin{proposition}\label{prp:class_monomial_psg} Let $G\subgr \SOr_n$ be a monomial absolutely irreducible rotation group that does not admit a two-dimensional system of imprimitivity. Then, up to conjugation, $G$ is one of the following groups
\begin{enumerate}
\item $M(\Pp_n)=M_n=D(n)\rtimes H_n$, $n=5,6,7,8$, for $H_n$ as in Theorem \ref{thm:class_perm_group}.
\item $M(\Qq_7)=M^p_7\subgr M_7$ and $M(\Qq_8)=M^p_8\subgr M_8$ as in Section \ref{sub:monomial_examples}. These groups are extensions of $\PSL_2(7)$ by a group of order $2^3$ and $2^7$, respectively.
\item An orientation preserving subgroup $\Wp$ of a reflection group $W$ of type $\BC_n$ or $\D_n$.
\end{enumerate}
\end{proposition}
For $n>4$ this result follows from \cite[Table I-III, p.~503]{MR0401936}, since all other complexified real groups occurring in these tables are either reducible (Group $1$, $e=g=\alpha=1$, in Table I, Group $1$, $e=g=\alpha=1$, in Table II, the first groups for $G=\AG_3(2)$ and $G=\PSL_2(7)$ with $e=g=\alpha=1$ in Table II, the second group for $G=\alt_5$ with $e=g=1$, $\alpha=-1$ in Table III and Group $G=\dih_5$, $e=g=h=1$ in Table III), conjugate to a group described in the proposition above (the conjugacy class in $\SOr_8$ of the second group for $G=\AG_3(2)$, $e=g=1$, in Table III is independent of the choice of $\alpha \in \{\pm 1\}$) or conjugate to a primitive rotation group (Group 3, $e=g=1$, $c=1$ in Table II is conjugate to the primitive rotation group $R_5(\alt_5)$) (note that the conditions $f=1$ and $e\in \{1,2\}$ must be satisfied in these tables in order for $G$ to be an orientation preserving real group). For arbitrary $n$ the result was independently obtained by working over the real numbers in \cite[p.~90]{MR608821}. In particular, for $n\leq 4$ a case differentiation shows that only the listed groups occur.\\

\textbf{Two-dimensional system of imprimitivity.} Now assume that $G$ admits a two-dimensional system of imprimitivity $\R^{n}=V_1\oplus \ldots \oplus V_l$ where $n=2l$. The block diagonal subgroup $D$ is the kernel of the natural homomorphism $\phi: G \To \sym_l$. Since $G$ is irreducible and generated by rotations, its image is a transitive subgroup of $\sym_l$ generated by transpositions and thus all of $\sym_l$ \cite[Lem. 2.13, p.~28]{MR2542964}. In particular, $G$ contains transformations of type
\[
	\begin{array}{ccl}
		 t_i=\left(
		  \begin{array}{cccccc}
		    I &  &  	&   &   & \\
		     & \ddots &  &  &  & \\
		     &  			& 0 & Q_i^{-1} &  & \\
		     &  			& Q_i & 0 &  & \\
		     &  			&   &   & \ddots & \\
		     &  			&   &  &  & I\\
		  \end{array}
		\right) \textrm{ with } t_iV_{i-1}=V_i
  \end{array}	
\]
for each $i=2,\ldots,l$ such that $G=\left\langle D, t_2,\ldots,t_l\right\rangle$ \cite[p.~511]{MR0401936}. Conjugating successively by the transformations $I\oplus Q_2^{-1} \oplus I \ldots \oplus I$, $I\oplus I \oplus (Q_3Q_2)^{-1} \oplus I \ldots \oplus I$, $\ldots$, we can assume that $Q_i=I$, $i=2,\ldots,l$. Each rotation in $g\in G$ is of one of the following four types (cf. \cite{Mikhailova2})

\begin{equation}\label{equ:type1}
		 g_{|V_i}= Q, \ g_{|V_i^{\bot}}=\mathrm{id}
\end{equation}
\begin{equation}\label{equ:type2}
g_{|V_i\oplus V_j} = \left(
		  \begin{array}{cc}
		    0 & Q^{-1}  \\
		    Q & 0 \\
		  \end{array}
		\right), \
g_{|(V_i\oplus V_j)^{\bot}}=\mathrm{id}
\end{equation}
\begin{equation}\label{equ:type3}
g_{|V_i\oplus V_j} = \left(
		  \begin{array}{cc}
		    R_1 & 0  \\
		    0 & R_2 \\
		  \end{array}
		\right), \
g_{|(V_i\oplus V_j)^{\bot}}=\mathrm{id}
\end{equation}
\begin{equation}\label{equ:type4}
g_{|V_i\oplus V_j} = \left(
		  \begin{array}{cc}
		    0 & R^{-1}  \\
		    R & 0 \\
		  \end{array}
		\right), \
g_{|(V_i\oplus V_j)^{\bot}}=\mathrm{id}
\end{equation}
for distinct $i,j \in \{1,\ldots,l\}$ and orthogonal matrices $Q$ with determinant $1$ and $R$, $R_1$ and $R_2$ with determinant $-1$, respectively. Note that if $G$ contains a rotation of type (\ref{equ:type4}), then it also contains a rotation of type (\ref{equ:type3}). Therefore, if $G$ does not contain a rotation of type (\ref{equ:type3}), then it preserves the complex structure $J:=J_0\oplus \ldots \oplus J_0$, where
\[
	\begin{array}{ccl}
		 	J_0= \left(
		  \begin{array}{cc}
		    0 & 1  \\
		    -1 & 0 \\
		  \end{array}\right)
  \end{array}	.
\]
In this case $G$ is induced by a unitary reflection group of type $G(m',p',l)$ for some $p'|m'$ (cf. Section \ref{sub:unitary}). Otherwise, each rotation in $G$ can be written as a composition of rotations of type (\ref{equ:type3}) and the $t_i$, $i=2,\ldots,l$, i.e. $G$ is generated by them. Moreover, since the $t_i$ normalize the set of rotations of type (\ref{equ:type3}), the diagonal subgroup $D$ is generated by these rotations. Let $G_i\subgr  \Or(V_i)$ be the projection of $G$ to $\Or(V_i)$, let $H_i\subgr \SOr(V_i)$ be the subgroup of $G$ generated by rotations of type (\ref{equ:type1}) contained in $G$ and set $H=H_1\times \ldots \times H_l \subgr G$. All $G_i$ are conjugate and isomorphic to a dihedral group $D_{km}$ of order $2km$ and all $H_i$ are conjugate and isomorphic to a cyclic group $C_{m}$ of order $m$. Assume that $Q\in G_1$ is a rotation of maximal order. Since the diagonal subgroup $D$ is generated by rotations of type (\ref{equ:type3}), there exists a rotation $Q'\in G_2$ such that 
\[
			Q\oplus Q' \oplus I \oplus \ldots \oplus I \in G.
\]
Let us now assume that $n>4$. Then, because of $l\geq3$, we also have
\begin{equation}\label{equ:Q_squared}
			Q\oplus Q^{-1} \oplus I \oplus \ldots \oplus I \in G.
\end{equation}
Since there exists a rotation $R_1\oplus I \oplus R_2 \oplus I \oplus \ldots \oplus I\in G$ of type (\ref{equ:type3}), we deduce that
\[
			QR_1Q^{-1}\oplus I \oplus R_2 \oplus \ldots \oplus I \in G
\]
and thus
\begin{equation}\label{equ:Q_Qinverse}
			Q^2 \oplus I \oplus \ldots \oplus I =QR_1Q^{-1}R_1\oplus I \oplus R_2^2 \oplus \ldots \oplus  \in G.
\end{equation}
This shows $k\in \{1,2\}$, that the subgroup of $G$ generated by the rotations of type (\ref{equ:type1}) and (\ref{equ:type2}) contained in $G$ is given by $\As(km,k,l)$ (cf. Section \ref{sub:imprimitive_example}) and that the group $G$ is generated by $G(km,k,l)$ and a transformation $r$ that conjugates the first two coordinates, i.e. $r(z_1,z_2,z_3\ldots,z_l)=(\overline{z}_1,\overline{z}_2,z_3\ldots,z_l)$, where we identify $\R^{2l}$ with $\C^l$. Hence, we have proven the following proposition.

\begin{proposition} \label{prp:class_imprimit_psg}
The imprimitive absolutely irreducible rotation groups $G\subgr \SOr_n$ for $n=2l\geq 5$ that admit a two-dimensional system of imprimitivity are up to conjugation $\Gs(km,k,l)=\left\langle G(km,k,l),\tau \right\rangle\subgr \SOr_n$ with $k=1,2$ and $km\geq3$. The group $\Gs(km,k,l)$ has order $2^{l-k} (km)^l l!$.
\end{proposition}
For $n=4$ there is no restriction on $k$ and for a specific $k$ there can be several geometrically inequivalent rotation groups (cf. Section \ref{sub:imprimitive_example}). More precisely, we have

\begin{proposition} \label{prp:class_imprimit_psg4}
The imprimitive absolutely irreducible rotation groups $G\subgr \SOr_4$ that admit a two-dimensional system of imprimitivity are precisely the unique extensions of reducible rotation groups $D$ defined by a set of data (cf. Theorem \ref{thm:clas_red_prg})
\[(\{(W(\I_2(km)),\Wp(\I_2(m)),W(\I_2(km))\}_{i\in \{1,2\}}, \varphi),\]
$km\geq 3$, where $\varphi: \dih_{k} \To \dih_{k}$ is an involutive automorphism of $\dih_k=W(\I_2(km)/\Wp(\I_2(m))$ that maps reflections onto reflections, by a normalizing rotation that interchanges the two irreducible components of $D$. They are denoted as $\Gs(km,k,2)_{\varphi}$ (cf. Section \ref{sub:imprimitive_example}).
\end{proposition}
\begin{proof} Let $G\subgr \SOr_4$ be an imprimitive absolutely irreducible rotation group as considered above. Then we have $G=\left\langle D, t_2\right\rangle$ where the block diagonal subgroup $D\subgr G$ is a reducible rotation group described by a set of data 
\[\{(W(\I_2(km)),\Wp(\I_2(m)),W(\I_2(km))\}_{i\in \{1,2\}},\]
$\varphi: \dih_{k} \To \dih_{k}$ where $\varphi$ is an automorphism of $\dih_{k}\cong W(\I_2(km))/\Wp(\I_2(m))$ that maps reflections onto reflections. Since $t_2$ normalizes $D$, the automorphism $\varphi$ has order $2$. Any other rotation that normalizes $D$ and interchanges its two irreducible components can be conjugated to $t_2$ by an element in the normalizer of $D$.
  \end{proof}

\subsection{Primitive rotation groups}
\label{sub:classification_primitive_pseudoreflection_groups}
In this section we prove the classification of primitive absolutely irreducible rotation groups in dimension $n\geq 5$. The complexification of a primitive absolutely irreducible rotation group is irreducible but a priori not primitive (the complexification of $R_5(\alt_5)$ is monomial). However, we are going to show that it satisfies the following property if $n\geq 5$.
\begin{definition}
\label{dfn:quasiprimitive}
An irreducible complex representation $\rho: G \To \GL(V)$ is called \emph{quasiprimitive} if for every normal subgroup $N$ of $G$ the restriction $\rho_{|N}$ splits into equivalent representations.
\end{definition}
Indeed, we have
\begin{lemma}\label{lem:quasiprimitive}
The complexified natural representation of a primitive absolutely irreducible rotation group $G\subgr \SOr_n$ with $n\geq 5$ is quasiprimitive.
\end{lemma}
\begin{proof}Let $N$ be a normal subgroup of $G$ and let $\R^n=V_1\oplus \ldots \oplus V_k$ be a decomposition into irreducible components with respect to the action of $N$. By Clifford's theorem \cite[Thm. 49.1, p.~343]{MR0144979} these irreducible components lie in one orbit for the action of $G$ on the equivalence classes of irreducible representations of $N$. Therefore, all of them are equivalent, because otherwise distinct isotypic components would define a nontrivial system of imprimitivity. Now the claim follows if we can show that the complexifications of the $V_i$ split into subrepresentations all of which are equivalent. If this were not the case, by the Frobenius-Schur theorem \cite[Thm. 4.7.3, p.~153]{MR928600} we would have $V_i^{\C}=U_i \oplus U_i^*$ for equivalent irreducible representations $U_i$, $i=1,\ldots,k$, which are inequivalent to $U_i^*$ and accordingly
\[
		\C^n = \underbrace{U_1\oplus \ldots \oplus U_k}_{=:\hat{U}} \oplus \underbrace{ U_1^*\oplus \ldots \oplus U_k^*}_{=:\hat{U}^*}.
\]
Then for any $g\in G$ we would either have $g\hat{U}=\hat{U}$ or $g\hat{U}=\hat{U}^*$. The second case yields a contradiction because $G$ is generated by rotations and since $\dim \hat{U} > 2$ holds by assumption. But $g\hat{U}=\hat{U}$ for all $g\in G$ also yields a contradiction since we have assumed $G$ to be absolutely irreducible. Consequently the complexified representation is quasiprimitive as claimed.
  \end{proof}

A nontrivial rotation group $G\subgr \SOr_n$ contains an element with eigenvalues $\xi,\bar{\xi},1,\ldots,1$, where $\xi$ is a nontrivial root of unity. We call such an element a special $r$-element if $\xi$ is an $r$-th root of unity. On the assumptions of this section the complexification of $G$ is quasiprimitive due to the preceding lemma. Finite quasiprimitive unimodular linear groups over the complex numbers in dimension higher than four that contain a special $r$-element are classified in \cite{MR0206088,MR0401937,MR0435243,MR0401936}. More precisely, in these papers possible quotient groups $G/Z_1$ are listed, where $Z_1$ is a subgroup of the center $Z$ of $G$. According to \cite[Thm. 1, p.~54]{MR0401937} it is sufficient to consider the cases $r=2$ and $r=3$, since the existence of a special $r$-element for any prime $r=p> 3$ implies $n\leq4$. The case $r=3$ is treated in \cite[Thm. 2, p.~261]{MR0401936} and the case $r=2$ is treated in \cite[Thm. 1, p. 58]{MR0435243} for $n\geq 6$ and in \cite[Thm. 9.A, p.~91]{MR0206088}, \cite[Table I-III]{MR0401936} for $n=5$.

Now we go through the cases and inspect which of the listed groups actually come from complexified primitive rotation groups, i.e. we examine the corresponding faithful complex representations. Such a representation can be excluded if it does not preserve the orientation or if it is not real meaning that it cannot be realized over the real numbers. The latter is in particular the case, if the restriction of the representation to a subgroup is not real or, by Schur's lemma, if the center $Z$ of $G$ has more than two elements. There is another convenient way to check whether an irreducible representation $\rho: G \To \GL_n(\C)$ is real or not. The Schur indicator of such a representation is defined as
\[
		\textrm{Ind}(\rho)=\frac{1}{|G|}\sum_{g\in G} \chi(g^2) 
\]
and it takes values in $\{1,0,-1\}$. Depending on its value the representation is said to be real, complex or quaternionic and only in the first case can it be realized over the real numbers \cite[p.~108]{MR0450380}. Also note that if $\rho_i: G_i \To \GL(V_i)$, $i=1,2$, are irreducible complex representations, then their tensor product $\rho: G_1 \times G_2 \To \GL(V_1\otimes V_2)$ is irreducible and its Schur indicator is given by $\textrm{Ind}(\rho)=\textrm{Ind}(\rho_1)\textrm{Ind}(\rho_2)$. Once we have found a real and orientation preserving representation we check whether the corresponding linear group is actually generated by rotations.

If $G$ is an absolutely irreducible rotation group its center is either trivial or $\{\pm 1\}$ by Schur's lemma and thus the same holds for $Z_1$. Therefore, in order to check if a given group $G_1=G/Z_1$ comes from a rotation group we only have to examine the representations of $G_1$ and of its two-fold central extensions. We will often be in a situation where $G_1$ is a perfect group. In this case its two-fold central extensions can be described as follows. Recall that the quotient of a perfect group by its center is centerless due to Gr\"un's Lemma \cite[p.~61]{MR1298629}.

\begin{lemma} \label{lem:perfect_extension} Let $G$ be a central extension of a perfect group $G_1$ by $Z_1\cong \Z_2$. Then one of the following two cases holds
\begin{enumerate}
\item $G$ is perfect and thus a perfect central extension of $G_1$.
\item $G\cong G_1 \times Z_1$
\end{enumerate}
Moreover, if $Z(G)=Z_1$, then in the first case the center of $G_1$ is trivial by Gr\"un's Lemma.
\end{lemma}
\begin{proof} Let $\pi:G \To G_1$ be the natural projection. Since $G_1$ is perfect we have $\pi(G')=G_1$, where $G'$ is the commutator subgroup of $G$. Therefore, the index of $G'$ in $G$ is either $1$ or $2$. If it is $1$ we have $G=G'$ and we are in case $(i)$. If it is $2$ we have $G'\cap Z_1 = \{1\}$ and hence $G=G' \times Z_1$ with $G'\cong G_1$ and we are in case $(ii)$.
  \end{proof}
Likewise the following lemma follows.
\begin{lemma}\label{lem:perfect_extension+auto} Let $G$ be a central extension of a group $G_1=\left\langle P_1, a \right\rangle$ which is generated by a perfect group $P_1$ and an automorphism $a$ of $P_1$ of order $2$ by $Z_1\cong \Z_2$. Let $P$ be the preimage of $P_1$ in $G$ and let $\tilde{a}$ be a preimage of $a$ in $G$. Then we have $G=\left\langle P, \tilde{a} \right\rangle$ and one of the following two cases holds.
\begin{enumerate}
\item $P$ is perfect and $Z_1\subgr P$.
\item $G\cong G_1 \times Z_1$
\end{enumerate}
\end{lemma}

Note that if $G_1$ in Lemma \ref{lem:perfect_extension} or $P_1$ in Lemma \ref{lem:perfect_extension+auto} is a simple group, then the irreducible representations of $G_1$ and $G$ can be looked up in many cases in \cite{MR827219}.

We begin by inspecting the possible groups in dimension five. The only irreducible complex five-dimensional linear group generated by elements with codimension two fixed-point subspace that is monomial and quasiprimitive is described in \cite{MR0401936}, Table II, Group $3$, $e=g=c=1$ (cf. \cite[Table I-III, p.~503]{MR0401936} and note that the diagonal subgroup $D$ of a monomial group $G$ can only consists of homotheties in order for $G$ to be quasiprimitive). This representation can also be realized over the real numbers and as such its image is the primitive absolutely irreducible rotation group $R_5(\alt_5)$ we have described in Lemma \ref{lem:exotic_psg_A5}. All other complexified primitive absolutely irreducible rotation groups in dimension $5$ must occur in the following list which we cite from \cite[Thm. 9.A, p.~91]{MR0206088}.

\begin{theorem} \label{thm:possible_G_dim5}
Let $\rho: G \To \mathrm{SL}_5(\C)$ be a faithful and irreducible representation of a finite group $G$ which is not monomial. Then one of the following cases holds.
\renewcommand\theenumi {\Alph{enumi}}
\begin{enumerate}
\item $G/Z \cong \PSL_2(11)$.
\item $G/Z$ is a symmetric or alternating group on five or six letters.
\item $G/Z \cong \Or_5(3)\cong \PSp_4(3) \cong \PSU_4(2)$.
\item $G$ is a uniquely determined group of order $24\cdot 5^4$ and has a nonabelian normal subgroup $N$ of order $125$ and exponent $5$.
\item $G$ is a certain subgroup of the group in $(D)$ that still contains $N$ as a normal subgroup.
\end{enumerate}
\end{theorem}

We can use this result to identify the primitive rotation groups in dimension five. 
\begin{proposition}\label{prp:primitive_dimension5}
The primitive absolutely irreducible rotation groups $G \subgr  \SOr_5$ are given, up to conjugation, as follows
\renewcommand\theenumi {\roman{enumi}}
\begin{enumerate}
\item The group $M(\Rr_5)=R_5(\alt_5)$ (cf. Lemma \ref{lem:exotic_psg_A5}).
\item The orientation preserving subgroup $\Wp$ of the reflection group $W$ of type $\A_5$.
\item The group $\Ws(\A_5)$ isomorphic to $\sym_6$ (cf. Proposition \ref{prp:extended_reflection_groups}, $(i)$).
\end{enumerate}
\end{proposition}
\begin{proof}
We go through the cases listed in Theorem \ref{thm:possible_G_dim5}. We can assume that the center $Z$ is trivial, since the dimension is odd and the orientation has to be preserved by $G$.

$(A)$ All five-dimensional irreducible representations of $\PSL_2(11)$ have Schur indicator $0$ \cite[p.~7]{MR827219} and thus this case can be excluded. 

$(B)$ The alternating group $\alt_5$ has one five-dimensional absolutely irreducible real representation \cite[p.~2]{MR827219}, which is described in Lemma \ref{lem:exotic_psg_A5}. This gives the rotation group in case $(i)$.

The symmetric group $\sym_5$ has one five-dimensional absolutely irreducible real representation, which is induced by the exceptional embedding $i:\sym_5 \To \sym_6$ \cite[p.~2]{MR827219} that maps a transposition in $\sym_5$ to a triple transposition in $\sym_6$. Therefore, this representation does not preserve the orientation and is thus not generated by rotations. Hence, we can exclude this case.

The alternating group $\alt_6$ has two inequivalent five-dimensional absolutely irreducible real representations, but they only differ by an outer automorphism of $\alt_6$ and thus give rise to the same linear group, namely the orientation preserving subgroup of the reflection group of type $\A_5$ \cite[p.~5]{MR827219}. This gives the rotation group in case $(ii)$.

The symmetric group $\sym_6$ has four inequivalent five-dimensional absolutely irreducible real representations, but for the same reason as above they only give rise to two different linear groups, the reflection group $W(\A_5)$ and the rotation group $\Ws(\A_5)$ described in Proposition \ref{prp:extended_reflection_groups}. This gives the rotation group in case $(iii)$.

$(C)$ All five-dimensional representations of $G \cong \Or_5(3)\cong \PSp_4(3) \cong \PSU_4(2)$ have Schur indicator $0$ \cite[p.~27]{MR827219} and thus this case can be excluded.

$(D)$ If $G$ were a complexified rotation group, it would be quasiprimitive by Lemma \ref{lem:quasiprimitive}  and thus the restriction of the representation to $N$ would either split into five equivalent one-dimensional representations or it would be irreducible. The first case cannot occur since one-dimensional representations of $N$ are not faithful. The second case cannot occur since the center of $N$ is divisible by $5$ which is why $N$ does not have faithful absolutely irreducible real representations. Hence this case can be excluded.

$(E)$ This case can be excluded by the same argument as in $(D)$.
  \end{proof}

Next we treat the case where $G$ contains a special $3$-element and where $n\geq 6$. These groups are listed in \cite[Thm. 2]{MR0401936} and in \cite[Thm. 1, case (A),(B),(H) cf. Rem. 1, p. 60]{MR0435243}. We first cite the results from \cite{MR0401936} and \cite{MR0435243}.

\begin{theorem}\label{thm:quasiprimitive-special3}
Let $\rho: G \To \mathrm{GL}_n(\C)$ be a faithful and quasiprimitive representation of a finite group $G$ with $n\geq6$ such that $\rho(G)$ contains a special $3$-element. Then one of the following cases holds.

(A) $G/Z = G_1$ where $G_1 \cong \alt_{n+1}$ or $G_1 \cong \sym_{n+1}$. All special elements lie in $\alt_{n+1}$ mod $Z$ and $G=G_1 \times Z$ if $G_1\cong \alt_{n+1}$, unless $n=6$.

(B) $G/Z_1 = G_1$ with $G_1\cong W(\E_n)$ or $G_1\cong \Wp(\E_n)$, $n=6,7,8$ and $Z_1\subgreq Z$.

(H) $n=6$, $G/Z \cong \PSU_4(3)$ or an extension by an automorphism of order $2$.

\end{theorem}
\begin{proof} Only the claim on the special elements in case (A) is not explicitly proven in \cite{MR0401936} and \cite{MR0435243}. However, if there were a special $2$-element not in $\alt_{n+1}$ mod $Z$, then the group would be listed in \cite[Table I, p.~63]{MR0435243}. The only possibility is the first row, where the involution is a transposition $(1,2)$ and the representation is the natural representation of the symmetric group. In particular, the involution is not a special $2$-element as it is a reflection. For a general special $r$-element $g \in  G_1$ we can assume that $r=2^a3^b$ by \cite[Thm. 1, p.~54]{MR0401937} and that $a<2$ by \cite[Thm. 1, p.~261]{MR0401936}. Because of $\sym_{n+1}/\alt_{n+1} \cong \Z_2$ all elements of odd order in $G_1$ are contained in $\alt_{n+1}$ and thus the special $3^b$-element $g^{2^a}$ is contained in $\alt_{n+1}$. The special element $g^{3^b}$ of order $2^a$ for $a<2$ is contained in $\alt_{n+1}$ by the argument above and thus so is $g$, since $2^a$ and $3^b$ are coprime.
  \end{proof}
Now we can identify the rotation groups appearing in Theorem \ref{thm:quasiprimitive-special3}.
\begin{proposition}\label{prp:primitive_dimension6}
The primitive absolutely irreducible rotation groups $G \subgr  \SOr_n$ for $n\geq 6$ that contain a special $3$-element are given, up to conjugation, as follows
\renewcommand\theenumi {\roman{enumi}}
\begin{enumerate}
\item The orientation preserving subgroups $\Wp$ of reflection groups $W$ of type $\A_n$, $\E_6$, $\E_7$ and $\E_8$.
\item The group $\Ws(\E_6)$ (cf. Proposition \ref{prp:extended_reflection_groups}, $(ii)$).
\end{enumerate}
\end{proposition}
\begin{proof}
We go through the cases listed in Theorem \ref{thm:quasiprimitive-special3}.

(A) Since all special elements lie in $\alt_{n+1}$ mod $Z$ we can assume that $G_1=\alt_{n+1}$. For  $n\geq 6$ the group $\alt_{n+1}$ has only one irreducible representation of dimension $n$ namely the nontrivial subrepresentation of the permutation representation on $\C^n$ \cite[Lem. 4.1, p.~273]{MR0401936}. This is a real representation and gives rise to the orientation preserving subgroup of the reflection group of type $\A_n$, i.e. the rotation group in $(i)$. The case $Z=\{\pm \mathrm{id}\}$, $G_1\cong \alt_{n+1}$ and $G = G_1 \times Z$ does not give new examples due to Lemma \ref{lem:reflection_extension}. It remains to consider the case where $G$ is a perfect central extension of $\alt_6$ by $Z\cong \Z_2$. Inspecting the character tables shows that there are no appropriate representations in this case \cite[p.~5]{MR827219}.

(B) The group $\Wp(\E_6)$ is isomorphic to the simple group $\Or_5(3)\cong \PSp_4(3) \cong \PSU_4(2)$ \cite[p.~27]{MR827219}. It has only one absolutely irreducible real $6$-dimensional representation namely its realization as the orientation preserving subgroup $\Wp(\E_6)$ of the reflection group of type $\E_6$. The double cover of $\Wp(\E_6)$ does not have representations in dimension $6$ \cite[p.~27]{MR827219}. According to Lemma \ref{lem:reflection_extension}, the group $\Ws(\E_6)\cong \Wp(\E_6) \times \Z_2$ is also a rotation group that occurs in this case. The group $W(\E_6)$ has two faithful absolutely irreducible real $6$-dimensional representations, among them its standard representation, and they differ only by a sign on the complement of $\Wp(\E_6)$ \cite[p.~27]{MR827219}. In particular, neither of them preserves the orientation. The double cover of $W(\E_6)$ does not have representations in dimension $6$ \cite[p.~27]{MR827219}.

The group $\Wp(\E_7)$ is isomorphic to the simple group $\PSp_6(2)$ \cite[p.~46]{MR827219}. It has only one absolutely irreducible real $7$-dimensional representation namely its realization as the orientation preserving subgroup $\Wp(\E_7)$  of the reflection group of type $\E_7$ \cite[p.~46]{MR827219}. Since the dimension is odd, the center must be trivial (the center of $W(\E_7)$ is not trivial, cf. \cite[p.~45]{MR1066460}) and thus the rotation group $\Wp(\E_7)$ is the only example that occurs in this case.

The group $\Wp(\E_8)$ is a perfect central extension of the simple group $\Or_8^+(2)$ by $\Z_2$ \cite[p.~85]{MR827219}. It has only one faithful absolutely irreducible real $8$-dimensional representation namely its realization as the orientation preserving subgroup $\Wp(\E_8)$ of the reflection group of type $\E_8$ \cite[p.~85]{MR827219} and its image contains the negative unit \cite[p.~46]{MR1066460}. For $Z=Z_1=\{\pm 1\}$ and $G\neq \Wp(\E_8) \times \Z_2$ the group $G$ must be perfect by Lemma \ref{lem:perfect_extension} and this contradicts Gr\"un's lemma, stating that the quotient of a perfect group by its center is centerless (cf. \cite[p.~61]{MR1298629}), since $\Wp(\E_8)$ has a nontrivial center. The group $W(\E_8)$ has two faithful absolutely irreducible real $8$-dimensional representations, but by the same reason as in $(B)$ neither of them preserves the orientation \cite[p.~85]{MR827219}. Suppose the group $W(\E_8)$ had a perfect central extension $G$ by $\Z_2$ with a suitable representation. Then the group $P$ (in the notation of Lemma \ref{lem:perfect_extension+auto}) would be a perfect central extension of $\Or_8^+(2)$ by a group $Z(P)$ of order $4$ containing $Z_1$ due to Lemma \ref{lem:perfect_extension+auto} and Gr\"un's Lemma. Therefore, the restriction of the representation of $G$ to $P$ would be reducible. Since the representation of $G$ is irreducible by assumption, the automorphism $\tilde{a}$ would permute two irreducible four-dimensional components of the representation of $P$. However, the linear group corresponding to such a representation cannot be generated by rotations. Hence, no further rotation groups occur in this case.

(H) There are no absolutely irreducible real $6$-dimensional representations in this case \cite[p.~53]{MR827219} and thus it can be excluded.
  \end{proof}

It remains to treat the cases where $G$ contains special $2$-elements but no special $r$-elements for $r\geq3$. We first cite the result obtained in \cite[Thm. 1, p.~58]{MR0435243}.
\begin{theorem}\label{thm:primitive_and_special_2}
Let $\rho: G \To \mathrm{GL}_n(\C)$ be a faithful and quasiprimitive representation of a finite group $G$ with $n\geq6$ such that $\rho(G)$ contains a special $2$-element but no special $r$-element for $r>2$. Then one of the following cases holds.
\renewcommand\theenumi {\Alph{enumi}}
\begin{enumerate}
\item As in case (A) in Theorem \ref{thm:quasiprimitive-special3}.
\item As in case (B) in Theorem \ref{thm:quasiprimitive-special3}.
\item $G/Z = A \times K$ where $K$ is a linear group generated projectively by reflections and $A\cong \alt_4$, $\sym_4$ or $\alt_5$. Here, $\rho(G)$ is a subgroup of $Y\otimes Y_1$ where $Y$ is a representation of degree $2$ and $Y_1$ is a representation of $K$ of degree $n/2$.
\item $n=6$ and $G/Z=G_1$ with $G_1 \cong \PSL_2(7)$ or an extension by an automorphism of order $2$ to $G_1 \cong \PGL_2(7)$ and if $G_1 \cong \PSL_2(7)$ then $G=G_1 \times Z$. 
\item $n=6,7$ and $G/Z=G_1$ with $G_1 \cong \PSU_3(3)$ or an extension by an automorphism of order $2$ to $G_1 \cong G_2(2)$. If $G_1\cong U_3(3) \cong \PSU_3(3)$ then $G=G_1 \times Z$. There is a unique representation in dimension $6$ and two representations in dimension $7$. The $7$-dimensional representations are not real and they do not extend to $G_2(2)$.
\item $n=6$ and $G/Z\cong \hat{J}_2$, a proper double cover of the Hall-Janko group of order $604800$.
\item $n=6$ and $G/Z=G_1$ with $G_1 \cong \PSL_3(4)$ or an extension by an automorphism of order $2$.
\item As in case (H) in Theorem \ref{thm:quasiprimitive-special3}.
\item $n=6$ and $G/Z_1=G_1 \cong \hat{\alt}_6$, the unique proper triple cover of $\alt_6$, or an extension by an automorphism of order $2$. Here $G_1$ has a center of order $3$ and $Z_1 \subgreq Z$. 
\item $n=8$ and $G$ contains a subgroup $G_1$ with $G=G_1Z$, $G_1\triangleright H$ where $H\cong Q_8 \circ \dih_4 \circ \dih_4$, $\dih_4 \circ \dih_4 \circ \dih_4$, or $\dih_4 \circ \dih_4 \circ \dih_4 \circ \Z_4$ and the restriction of the representation to $H$ is the tensor product of faithful representations of the quaternion group $Q_8$ and the dihedral group $\dih_4$ with $|\dih_4|=8$ on $\C^2$ and the cyclic group $\Z_4$ on $\C$ (cf. \cite[Theorems~2.7.1 and 2.7.2]{MR0231903}). The quotient $G_1/H$ is isomorphic to a subgroup of $\Or_6^+(2)\cong \sym_8$, $\Or^-_6(2)$ or $S_p(6,2)$ in the respective cases and $\rho_{|H}$ is irreducible.
\item $n=8$ and $G/Z\cong (\alt_5\times \alt_5 \times \alt_5)\wr \sym_3$. $G$ contains a normal subgroup $H\cong \SL_2(5) \circ \SL_2(5) \circ \SL_2(5)$ and $\rho_{|H}= \rho_1 \otimes \rho_1 \otimes \rho_1$ for a two dimensional representation $\rho_1$ of $\SL_2(5)$.
\item $n=10$ and $G=G_1\times Z$ with $G_1=\PSU_5(2)$.
\item $n=6$ and $G=G_1\circ Z$ where $G_1$ is a proper central extension of $\alt_7$ with a center of order $3$.
\end{enumerate}
\renewcommand\theenumi {\roman{enumi}}
\end{theorem}
\begin{proposition}\label{prp:primitive_dimension6_conlusion}
The primitive absolutely irreducible rotation groups $G \subgr  \SOr_n$ for $n\geq 6$ that contain no special $r$-element for $r\geq3$ are given, up to conjugation, by $M(\Ss_6)=R_6(\PSL_2(7))$ (cf. Lemma \ref{lem:exotic_psg_psl}) and $M(\Tt_8)=L$ (cf. Section \ref{sec:group_L_sec}).
\end{proposition}
\begin{proof} We go through the cases listed in Theorem \ref{thm:primitive_and_special_2}. For (A), (B) and (H) the group contains special $3$-elements (cf. \cite[\emph{Remark} on p.~60]{MR0435243}).  These cases have already been treated in Proposition \ref{prp:primitive_dimension6}.

(C) It follows from the eigenvalue structure of an element  $A\otimes B \in Y \otimes Y_1$ and from $n/2\geq 3$ that all special $r$-elements in $G$ are contained in $K$ mod $Z$. In particular, $G$ cannot be generated by special $r$-elements and hence this case can be excluded.

(D) The simple group $\PSL_2(7)$ has only one 6-dimensional absolutely irreducible real representation \cite[p.~3]{MR827219} which gives rise to the rotation group $R_6(\PSL_2(7))$. It extends to two 6-dimensional absolutely irreducible real representations of $\PGL_2(7)$. According to their character tables all special $2$-elements lie in $\PSL_2(7)$ \cite[p.~3]{MR827219}. By Lemma \ref{lem:exotic_psg_psl}, the group $\left\langle R_6(\PSL_2(7)), -1 \right\rangle$ is not generated by rotations. Since there are no other appropriate representations \cite[p.~3]{MR827219} the rotation group $R_6(\PSL_2(7))$ is the only example that can occur in this case.

(E) None of the representations in question is real \cite[p.~14]{MR827219} and thus this case can be excluded.

(F) The group $\hat{J}_2$ does not have an absolutely irreducible real representations in dimension $6$ \cite[p.~43]{MR827219}. The case $Z=\{\pm 1\}$ and $G\neq \hat{J}_2 \times Z$ cannot occur, since then $G$ would be perfect by Lemma \ref{lem:perfect_extension} contradicting Gr\"un's lemma. Hence, no examples occur in this case.

(G) There are no $6$-dimensional absolutely irreducible real representation in this case \cite[p.~23]{MR827219} and thus it can be excluded.

(I) For $G= G_1 \times Z_1$ we have $Z(G)\geq3$, i.e. no rotation group can occur. The case $Z=Z_1=\{\pm 1\}$ with $G\neq G_1 \times Z$ is impossible by Gr\"un's lemma, since $G_1$ has a nontrivial center (cf. Lemma \ref{lem:perfect_extension}).

(J) The Schur indicators of the listed possible representations of $H$ are $-1$, $1$ and $0$ and thus only the second case comes into question. In this case the representation is real and its image in $\SOr_8$ is given by the group $H<\SOr_8$ described in Section \ref{sec:group_L_sec}. We have to look for primitive groups in the normalizer $N_{\GL(8,\R)}(H)$ that are generated by pseudoreflections, contain the group $H$ as a subgroup and only pseudoreflections of order $2$. By Schur's lemma, it suffices to look for rotation groups in $N_{\SOr_8}(H)$ with these properties. Therefore, according to Lemma \ref{lem:pssgroups_L} and Lemma \ref{lem:rot_contained_in_L} the group $L$ defined in Section \ref{sec:group_L_sec} is the only example that occurs in this case.

(K) There are two faithful representation of $\SL_2(5)$ in dimension two, both have Schur indicator $-1$ \cite[p.~2]{MR827219} and thus so has the irreducible representation $\rho_{|H}$. In  particular, the representation of $G$ is not real and hence no examples occur in this case.

(L) There are no absolutely irreducible real $10$-dimensional representation in this case \cite[p.~72]{MR827219} and thus it can be excluded.

(M) Because of $|Z|\geq3$ the representation of $G$ cannot be real and thus this case can be excluded.
  \end{proof}

We summarize what we have obtained in this section.

\begin{proposition}\label{prp:real_primitive_prg}
The primitive absolutely irreducible rotation groups $G \subgr  \SOr_n$ for $n\geq 5$ are given, up to conjugation, as follows.
\renewcommand\theenumi {\roman{enumi}}
\begin{enumerate}
\item The orientation preserving subgroups $\Wp$ of the reflection groups of type $\A_n$, $\E_6$, $\E_7$ and $\E_8$.
\item The group $M(\Rr_5)=R_5(\alt_5)\subgr \SOr_5$ (cf. Lemma \ref{lem:exotic_psg_A5}).
\item The group $\Ws(\A_5)\subgr \SOr_5$ (cf. Proposition \ref{prp:extended_reflection_groups}, $(i)$).
\item The group $M(\Ss_6)=R_6(\PSL_2(7))\subgr \SOr_6$  (cf. Lemma \ref{lem:exotic_psg_psl}).
\item The group $\Ws(\E_6)\subgr \SOr_6$ (cf. Proposition \ref{prp:extended_reflection_groups}).
\item The group $M(\Tt_8)=L\subgr \SOr_8$ (cf. Section \ref{sec:group_L_sec}).
\end{enumerate}
\end{proposition}

\section{Irreducible reflection-rotation groups}
\label{sec:class_irr_rofl}
Let $G\subgr  \Or_n$ be an irreducible reflection-rotation group. The case in which $G$ is generated by rotations is subject of Theorem \ref{thm:clas_red_psg}. So suppose that $G$ contains a reflection. Let $F$ be the normal subgroup of $G$ generated by the reflections in $G$ and let $H$ be the normal subgroup of $G$ generated by the rotations in $G$. Then $H$ is the orientation preserving subgroup of $G$ and it is absolutely irreducible for $n>2$ by Lemma \ref{lem:orient_is_abs_irr}.

\begin{proposition}\label{prp:irr_ref_ro_red} Let $G<\Or_n$ be an irreducible reflection-rotation group that contains a reflection such that $F$ is reducible. Then $G$ is either one of the monomial groups $\Mt_5$, $\Mt_6$, $\Mt_7$, $\Mt_8$, $\Mt(\D_n)$ (cf. Section \ref{sub:monomial_examples} and Table \ref{tab:ref-rot4}) or an imprimitive group $\Gt (km,k,l)\subgr \SOr_{2l}$ with $k=1,2$ and $km\geq3$ (cf. Section \ref{sub:imprimitive_example} and Section \ref{sub:imprimitive}).
\end{proposition}
\begin{proof} Observe that $F$ is distinct from $G$ and thus we have $n>2$ and $H$ is absolutely irreducible. Since $G$ is irreducible, the group $H$ permutes the irreducible components of $F$ transitively. Therefore, the rotation group $H$ is imprimitive with a system of imprimitivity given by the irreducible components of $F$ which are all equivalent and either one- or two-dimensional. If they are one-dimensional, then $H$ is a monomial group occurring in Proposition \ref{prp:class_monomial_psg} that contains all transformations that change the sign of an even number of coordinates. Hence, the group $G$, being not a reflection group by assumption, is one of the listed monomial groups in this case.

In the second case, the Coxeter diagram of $F$ is given by
\[
			\bullet_{s^{(1)}_1}\stackrel{m}{-}\bullet_{s^{(1)}_2} \ \ \  \bullet_{s^{(2)}_1}\stackrel{m}{-}\bullet_{s^{(2)}_2}\\ \dots \\ \bullet_{s^{(l)}_1}\stackrel{m}{-}\bullet_{s^{(l)}_2}
\]
with $m>2$ and $l>1$. As in the proof of Proposition \ref{prp:class_imprimit_psg} we see that $H$ acts like the symmetric group on the irreducible components of $F$. Since the orientation preserving subgroup of $F$ is contained in $H$, this implies $\Gs(m,1,l) < H$. Moreover, the fact that $H$ normalizes $F$ implies $H<\Gs(2m,2,l)$. Therefore, $H$ is an imprimitive rotation group of type $\Gs(km,k,l)$ for $k=1$ or $k=2$ and $km\geq3$ by Proposition \ref{prp:class_imprimit_psg} and Proposition \ref{prp:class_imprimit_psg4}. Accordingly, $G$ is an imprimitive reflection-rotation group of type $\Gt (km,k,l)\subgr \SOr_{2l}$ for $k=1$ or $k=2$ and $km\geq3$ in this case (cf. Section \ref{sub:imprimitive_example}).
  \end{proof}

\begin{proposition} Let $G$ be an irreducible reflection-rotation group that contains a reflection such that $F$ is irreducible and distinct from $G$. Then $G$ is a group of type $\Wt$ generated by an irreducible reflection group $W$ of type $\A_4$, $\D_4$, $\F_4$, $\A_5$ or $\E_6$ and a normalizing rotation (cf. Lemma \ref{prp:extended_psrg}).
\end{proposition}
\begin{proof} By assumption there is a rotation $h \in G \backslash F$. According to Lemma \ref{lem:existence_chamber} there exists a chamber of the reflection group $F$ such that $hC=C$ and by Lemma \ref{lem:implication_of_fixed_chamber} we deduce that $F$ has type $\A_4$, $\D_4$, $\F_4$, $\A_5$ or $\E_6$. Finally, the uniqueness statement of Lemma \ref{prp:extended_psrg} implies that the group $G$ is generated by $F$ and $h$ and is thus one of the listed groups. 
  \end{proof}

\section{General reflection-rotation groups}
\label{sec:class_reducible}
The structure of reducible rotation groups is described in \cite{Mikhailova2}. For a general reflection-rotation group $G\subgr \Or_n$ let $\R^n=V_1 \oplus \ldots \oplus V_k$ be a decomposition into irreducible components. For each $i \in I$ we denote the projection of $G$ to $\Or(V_i)$ by $\pi_i$ and set $G_i = \pi_i(G)$. Recall the following definition from the introduction.

\begin{definition}
A rotation $g\in G$ is called a rotation of the
\begin{enumerate}
	\item \emph{first kind}, if for some $i_0 \in I$, $\pi_{i_0}(g)$ is a rotation in $V_{i_0}$ and $\pi_i(g)$ is the identity on $V_{i}$ for all $i\neq i_0$.
	\item \emph{second kind}, if for some $i_1,i_2 \in I$, $i_1 \neq i_2$, $\pi_{i_1}(g)$ and $\pi_{i_2}(g)$ are reflections in $V_{i_1}$ and $V_{i_2}$, respectively, and $\pi_i(g)$ is the identity for all $i\neq i_1,i_2$.
\end{enumerate}
\end{definition}

Let $H$ be the normal subgroup of $G$ generated by rotations of the first kind, let $F$ be the normal subgroup of $G$ generated by reflections and rotations of the second kind and set $H_i=\pi_i(H)$ and $F_i=\pi_i(F)$. Then $H_i$ is a rotation group, $F_i$ is a reflection group, both are normal subgroups of $G_i$ and $G_i$ is generated by them. The triple $(G_i,H_i,F_i)$ has an additional property that does not hold in general. It is described in Lemma \ref{lem:red6} below. 

\begin{lemma}\label{lem:red2}
For every reflection $s\in F_i$ there exists a reflection or a rotation of the second kind $g\in G$ such that $s=\pi_i(g)$.
\end{lemma}
\begin{proof} Let $X_i$ be the set of reflections in $F_i$ of the form $\pi_i(g)$ for some reflection or rotation of the second kind $g\in G$. Then $X_i$ generates $F_i$ and is invariant under conjugation by $F_i$. Therefore, every reflection in $F_i$ is contained in $X_i$, i.e. is a reflection of the form $s=\pi_i(g)$ for some reflection or rotation of the second kind $g\in G$ \cite[Prop. 1.14, p.~24]{MR1066460}.
  \end{proof}

\begin{lemma}\label{lem:red6}
Let $\tau \in F_i$ be a reflection and let $h \in H_i$. If $h\tau$ is a reflection, then it is contained in $F_i$.
\end{lemma}
\begin{proof} By Lemma \ref{lem:red2} there exists a reflection or a rotation of the second kind $g\in G$ such that $\tau=\pi_i(g)$. Then $hg$ is either a reflection or a rotation of the second kind in $G$ and $h\tau=\pi_i(hg)$ is contained in $F_i$.
  \end{proof}

\begin{lemma}\label{lem:comjugate_refl}
Let $s, \tau \in F_i$ be reflections conjugate under $H_i$, i.e. $\tau = h s h^{-1}$ for some $h\in H_i$. Then we have $s\tau \in H_i$.
\end{lemma}
\begin{proof} The claim follows from $s\tau = s h s h^{-1}= (shs^{-1})h^{-1}$, since $H_i$ is normal in $G_i$.
  \end{proof}

In the following all possible triples $(G_i,F_i,H_i)$ are described. Notice that we actually classify triples $(G_{rr},M,W)$ where $G_{rr}$ is an irreducible reflection-rotation group, $M$ a rotation group and $W$ a reflection group such that the properties stated in Remark 1 hold. If $F_i$ is trivial, then $G_i=H_i\subgr G$ is an irreducible rotation group and splits off as a direct factor. Otherwise, $G_i$ is one of the irreducible reflection-rotation groups we have described in the preceding section. Let $S=\{s_1,\ldots,s_l\}\subgr F_i$ be a set of simple reflections generating $F_i$ \cite[p.~10]{MR1066460}. We denote the image of a reflection $s \in G_i$ in $G_i/H_i$ by $\overline{s}_i$. Since $G_i$ is generated by $F_i$ and $H_i$, the quotient group $G_i/H_i$ is generated by the set $\overline{S}$ composed of the different cosets among the $\overline{s}_1,\ldots,\overline{s}_l$. We have $G_i/H_i\cong F_i/\tilde{H}_i$ with $\tilde{H}_i= H_i \cap F_i$ and we will see that in each case $\tilde{H}_i$ is generated by the conjugates of elements of the form $(s_rs_s)^{\tilde{m}_{rs}}$ with $\tilde{m}_{rs}\leq m_{rs}$ where $m_{rs}$ are the entries of the Coxeter matrix of $F_i$. It is then clear that $(G_i/H_i,\overline{S})$ is a Coxeter system with Coxeter matrix obtained by removing the redundant entries in $(\tilde{m}_{rs})$. We say that an element in a Coxeter group is a reflection, if it is conjugate to a generator or, equivalently, if its image under the geometric representation is a reflection \cite[p.~108]{MR1066460}. It will then follow directly that the reflections in $G_i/H_i$ are precisely the cosets of reflections in $F_i$ (cf. Corollary \ref{cor:cosets_of_reflection}).

For the proof of Theorem \ref{thm:reducible_pairs} three different cases are considered.

\begin{proposition}\label{prp:red1} Assume that $F_i$ is a nontrivial reducible reflection group. Then $H_i$ is an imprimitive rotation group and a set of simple reflections generating $F_i$ projects onto a set $\overline{S}\subset G_i/H_i$ for which $(G_i/H_i,\overline{S})$ is a Coxeter system of type $\A_1$ or $\A_1\times \A_1$. More precisely, the triple $(G_i,F_i,H_i)$ occurs in one of the cases $(i)$ to $(iii)$ in Theorem \ref{thm:reducible_pairs}.
\end{proposition}

\begin{proof} As in Proposition \ref{prp:irr_ref_ro_red}, the group $H_i$ is an imprimitive rotation group with a system of imprimitivity given by the irreducible components of $F_i$, which are all equivalent and either one- or two-dimensional. If they are one-dimensional, then, given Lemma \ref{lem:comjugate_refl}, it follows as in the proof of Proposition \ref{prp:irr_ref_ro_red} that $G_i$ is one of the monomial groups $\Mt$ listed in Proposition \ref{prp:irr_ref_ro_red} whose orientation preserving subgroup is $H_i$. In particular, $G_i/H_i$ is a Coxeter group of type $\A_1$ and we are in case $(i)$ of Theorem \ref{thm:reducible_pairs}.

In the second case, $F_i$ has the Coxeter diagram
\[
			\bullet_{s^{(1)}_1}\stackrel{m_0}{-}\bullet_{s^{(1)}_2} \ \ \  \bullet_{s^{(2)}_1}\stackrel{m_0}{-}\bullet_{s^{(2)}_2}\\ \dots \\ \bullet_{s^{(l)}_1}\stackrel{m_0}{-}\bullet_{s^{(l)}_2}
\]
with $m_0>2$ and $l>1$. As in Proposition \ref{prp:irr_ref_ro_red} we see that $H_i< \Gs(2m_0,2,l)$ and that $H_i$ acts on the irreducible components of $F_i$ as the symmetric group $\sym_l$ (cf. Proposition \ref{prp:class_imprimit_psg} and Proposition \ref{prp:class_imprimit_psg4}). We can choose the generators of $F_i$ such that all $s_1^{(j)}$ and all $s_2^{(j)}$ are conjugate among each other under $H_i$ and thus identical in $G_i/H_i$ by Lemma \ref{lem:comjugate_refl}. Let $k$ be the smallest positive integer such that $(s_1^{(j)}s_2^{(j)})^{k} \in H_i$. Then $k$ divides $m_0$ and for $m=\frac{m_0}{k}$ we have $\Gs(mk,k,l)<H_i$, because the rotations $s_1^{(j)}s_1^{(j')}$, $s_2^{(j)}s_2^{(j')}$, $j,j'=1,\ldots,l$, are contained in $H_i$ (cf. Section \ref{sub:imprimitive}). If $H_i= \Gs(2m_0,2,l)$, then $k=1$ and $G=\Gt(2m_0,2,l)$. Otherwise, we have $H_i< \Gs(m_0,1,l)$. For $k=1$ this implies $H_i= \Gs(m_0,1,l)$ and $G=\Gt(m_0,1,l)$. For $k\neq 1$ the relation
\[
		k=\mathrm{ord}(\overline{s}^{(1)}_1\overline{s}_2^{(1)})=\mathrm{ord}(\overline{s}_1^{(1)}\overline{s}_2^{(2)}) \leq 2
\]
shows that $k=2$ and hence $H_i= \Gs(m_0,1,l)$ and $G=\Gt(m_0,1,l)$. In each case the group $\tilde{H}_i=H_i \cap F_i$ is generated by the rotations $(s_1^{(j)}s_2^{(j)})^{k}$, $s_1^{(j)}s_1^{(j')}$, $s_2^{(j)}s_2^{(j')}$, $j,j'=1,\ldots,l$. Thus $G_i/H_i$ is either a Coxeter group of type $\A_1$ or $\A_1\times \A_1$ depending on whether $k=1$ or $k=2$ and we are in case $(ii)$ or $(iii)$ of Theorem \ref{thm:reducible_pairs}.
  \end{proof}
For the other two cases we need the following two facts on reflection groups.
\begin{lemma}\label{lem:red7}
Let $s_1,\ldots,s_l$ be simple reflections generating a reflection group $W$ and let $M<W$ be a rotation group. Then every rotation $h \in M$ is conjugate to $(s_is_j)^r$ for some $i,j \in \{1,\ldots,l\}$ and some positive integer $r$. In particular, if $s_is_j$ has prime order, then $h'=s_is_j$ is a rotation contained in $M$. 
\end{lemma}
\begin{proof} The linear fixed point subspace $U = \mathrm{Fix}(h)$ of $h$ has codimension two and is contained in a hyperplane corresponding to a reflection $s\in W$, since $W$ acts freely on its chambers \cite[p.~23]{MR1066460}. The composition $s'=sh$ is another reflection in $W$ whose linear fixed point subspace contains $U$. Let $s''\in W$ be a reflection different from $s$ with  $U\subset\mathrm{Fix}(s'')$ such that $s$ and $s''$ are faces of a common chamber. Then we have $h \in \left\langle ss''\right\rangle$ and thus the claim follows, since all sets of generating simple reflections in a reflection group are conjugate to each other \cite[Thm. 1.4, p.~10]{MR1066460}.
  \end{proof}

\begin{lemma}\label{lem:red5} Let $s_1,s_2,\tau \in W$ be reflections in a reflection group $W$ and let $M\triangleleft W$ be a normal subgroup generated by rotations such that $\overline{s}_1=\overline{s}_2 \in W/M$ and set $m=\mathrm{ord}(s_1\tau)$ and $n=\mathrm{ord}(s_2\tau)$. Then for $d=\mathrm{gcd}(m,n)$ the powers $(s_1\tau)^d$ and $(s_2\tau)^d$ are contained in $M$. In particular, $d=1$ implies $\overline{s}_1=\overline{s}_2=\overline{\tau}$.
\end{lemma}
\begin{proof} Choose integers $p,q$ such that $d=mp+nq$. Because of $\overline{s}_1=\overline{s}_2$ we have
\[
		(\overline{s}_1\overline{\tau})^d=(\overline{s}_2\overline{\tau})^d=(\overline{s}_1\overline{\tau})^{mp}(\overline{s}_2\overline{\tau})^{nq}=e
\]
and thus $(s_1\tau)^d, (s_2\tau)^d \in M$.
  \end{proof}

\begin{proposition}\label{prp:red2} Assume that $G_i=F_i$ is an irreducible reflection group. Then a set of simple reflections generating $F_i$ projects onto a set $\overline{S}\subset G_i/H_i$ for which $(G_i/H_i,\overline{S})$ is a Coxeter system. More precisely, the quadruple $(G_i,F_i,H_i,\Gamma_i)$ occurs in one of the cases $(iv)$ to $(xii)$ in Theorem \ref{thm:reducible_pairs}.
\end{proposition}
\begin{proof}
Let $\{s_1,\ldots,s_l\}\subgr G_i$ be a set of simple reflections generating $G_i$ and set $\overline{m}_{ij}=\mathrm{ord}(\overline{s}_i\overline{s}_j)$. According to Lemma \ref{lem:red7} the group $H_i$ is generated by conjugates of elements of the form $(s_rs_s)^{\tilde{m}_{rs}}$ with $\tilde{m}_{rs}\leq m_{rs}$ and thus $(G_i/H_i,\overline{S})$ is a Coxeter system.

For trivial $H_i$ the quotient $G_i/H_i$ can be any irreducible Coxeter group and we are in case $(iv)$ of Theorem \ref{thm:reducible_pairs}. If all generators lie in the same coset of $H_i$, then $H_i$ is the orientation preserving subgroup of the reflection group $G_i$ by Lemma \ref{lem:comjugate_refl} and the quotient group $G_i/H_i \cong \Z_2$ is generated by the coset of a reflection in $G_i$. Hence we are case $(v)$ of Theorem \ref{thm:reducible_pairs}. If $H_i$ is nontrivial, then Lemma \ref{lem:red7} implies that $(s_is_j)^r \in H_i$ for some pair of distinct generators $s_i$ and $s_j$ and some $r<\mathrm{ord}(s_is_j)$. If we additionally assume that not all generators of $G_i$ lie in the same coset of $H_i$, then Lemma \ref{lem:red5} implies that only the types $\A_3$, $\BC_n$, $\D_n$, $\I_2(m)$ and $\F_4$ can occur for $G_i$. More precisely, the following cases can occur.\newline
(A) $G_i = W(\A_3)$.
\[
			\bullet_{s_1}-\bullet_{s_2}-\bullet_{s_3}
\]
We have $\overline{s}_1=\overline{s}_3\neq \overline{s}_2$. The group $G_i$ is the symmetry group of a tetrahedron and $H_i=\Wp(A_1\times A_1 \times A_1)$ is its unique orientation preserving normal subgroup isomorphic to $\Z_2 \times \Z_2$. The quotient group $G_i/H_i$ has the Coxeter diagram
\[
			\circ_{\overline{s}_1}-\circ_{\overline{s}_2}
\]
(B) $G_i = W(\BC_l)$, $l\geq3$.
\[
			\bullet_{s_1}-\bullet_{s_2}- \dots \bullet_{s_{l-1}}=\bullet_{s_l}
\]
In any case we have $\overline{m}_{l-1,l}=2$ by Lemma \ref{lem:red7} and Lemma \ref{lem:red5} and thus $H_i$ contains the diagonal subgroup of $\Wp(\BC_l)$. If all generators of $G_i$ lie in different cosets of $H_i$, then $H_i=D(W(\D_l))$ and the quotient group $G_i/H_i$ has the Coxeter diagram
\[
			\circ_{\overline{s}_1}-\circ_{\overline{s}_2}- \dots \circ_{\overline{s}_{l-1}} \ \ \ \circ_{\overline{s}_l}
\]
Otherwise, for $l\neq 4$ Lemma \ref{lem:red5} implies $\overline{s}_1=\ldots=\overline{s}_{l-1}\neq\overline{s}_{l}$. In this case we have $H_i=\Wp(\D_l)$ and the quotient group $G_i/H_i$ has the Coxeter diagram
\[
			\circ_{\overline{s}_1}  \ \ \ \circ_{\overline{s}_l}
\]
For $l=4$ we may also have $\overline{s}_1=\overline{s}_3\neq \overline{s}_2$ and $\overline{s}_1,\overline{s}_2\neq \overline{s}_4$. In this case $H_i$ is the preimage in $\Wp(\BC_4)=(\mathbf{O}/\mathbf{V};\mathbf{O}/\mathbf{V})$ of the normal subgroup of $\sym_4\cong \Wp(\BC_4)/D(\Wp(\BC_4))$ isomorphic to $\Z_2 \times \Z_2$ (cf. case (A)). It is a monomial rotation group of type $\Gs(4,2,2)=(\mathbf{V}/\mathbf{V};\mathbf{V}/\mathbf{V})$. The quotient group $G_i/H_i$ has the Coxeter diagram
\[
			\circ_{\overline{s}_1}-\circ_{\overline{s}_2}  \ \ \ \circ_{\overline{s}_4}
\]
(C) $G_i= W(\D_l)$, $l\geq4$.
\[
			\bullet_{s_1}-\bullet_{s_2}- \dots \bullet_{s_{l-2}} < 
		  \begin{array}{l}
		    \bullet_{s_l} \\
		    \bullet_{s_{l-1}} \\
		  \end{array}
\]
In any case we have $\overline{s}_ {l-1}=\overline{s}_ {l}$ (perhaps after relabeling in the case $l=4$) by Lemma \ref{lem:red7} and Lemma \ref{lem:red5} and thus $D(W(\D_l))<H_i$. For $l\neq 4$ all other generators lie in different cosets of $H_i$. In this case we have $H_i=D(W(\D_l))$ and the Coxeter diagram of the quotient group $G_i/H_i$ is
\[
	\circ_{\overline{s}_1}-\circ_{\overline{s}_2}- \dots \circ_{\overline{s}_{l-2}} - \circ_{\overline{s}_{l-1}}
\]
For $l=4$ we may also have $\overline{s_1}=\overline{s_3}=\overline{s_4}\neq\overline{s_2}$. In this case $H_i=\Gs(4,2,2)=(\mathbf{V}/\mathbf{V};\mathbf{V}/\mathbf{V})$ (cf. case (B)) and the quotient group $G_i/H_i$ has the Coxeter diagram
\[
	\circ_{\overline{s}_1}-\circ_{\overline{s}_2}
\]
(D) $G_i = W(\I_2(m))$ for $m\geq 4$.
\[
			\bullet_{s_1}\stackrel{m}{-}\bullet_{s_2}
\]
We have $\overline{s}_1\neq \overline{s}_2$ and $H_i$ is a cyclic group of order $\frac{m}{\overline{m}_{1,2}}$. Consequently, the quotient group $G_i/H_i$ is a dihedral group of type $\I_2(\overline{m}_{1,2})$ with Coxeter diagram
\[
			\circ_{s_1}\stackrel{\overline{m}_{1,2}}{-}\circ_{s_2}
\]
(E) $G_i= W(\F_4)$.
\[
			\bullet_{s_1}-\bullet_{s_2}=\bullet_{s_3}-\bullet_{s_4}
\]
In any case we have $\overline{m}_{2,3}=2$ by Lemma \ref{lem:red7} and Lemma \ref{lem:red5}. If all generators lie in different cosets of $H_i$, then $H_i=\Gs(4,2,2)=(\mathbf{V}/\mathbf{V};\mathbf{V}/\mathbf{V})$ and the quotient group $G_i/H_i$ has the Coxeter diagram
\[
			\circ_{\overline{s}_1}-\circ_{\overline{s}_2} \ \ \ \circ_{\overline{s}_3}-\circ_{\overline{s}_4}
\]
If $\overline{s}_3=\overline{s}_4$ and all other generators lie in different cosets of $H_i$, then $H_i=\Wp(\D_4)=(\mathbf{T}/\mathbf{V};\mathbf{T}/\mathbf{V})$ and the quotient group $G_i/H_i$ has the Coxeter diagram
\[
			\circ_{\overline{s}_1}-\circ_{\overline{s}_2} \ \ \ \circ_{\overline{s}_3}
\]
Finally, if $\overline{s}_1=\overline{s}_2$, $\overline{s}_3=\overline{s}_4$, then $H_i=\Ws(\D_4)=(\mathbf{T}/\mathbf{T};\mathbf{T}/\mathbf{T})$ and the quotient group $G_i/H_i$ has the Coxeter diagram
\[
			\circ_{\overline{s}_1} \ \ \ \circ_{\overline{s}_3}
\]
  \end{proof}

\begin{proposition}\label{prp:red3} Assume that $F_i$ is an irreducible reflection group different from $G_i$. Then a set of simple reflections generating $F_i$ projects onto a set $\overline{S}\subset G_i/H_i$ for which $(G_i/H_i,\overline{S})$ is a Coxeter system of type $\A_1$ or $\A_1\times \A_1$. More precisely, the quadruple $(G_i,F_i,H_i,\Gamma_i)$ occurs in one of the cases $(xvi)$ to $(xvii)$ in Theorem \ref{thm:reducible_pairs}.
\end{proposition}
\begin{proof} Let $h \in H_i \backslash F_i$ be a rotation. By Lemma \ref{lem:red6} and the proof of Lemma \ref{lem:existence_chamber} there exists a chamber of the reflection group $F_i$ such that $hC=C$. By Lemma \ref{lem:implication_of_fixed_chamber} we deduce that $F_i$ has type $\A_4$, $\D_4$, $\F_4$, $\A_5$ or $\E_6$. In the cases of $\A_4$, $\F_4$, $\A_5$ and $\E_6$ Lemma \ref{lem:comjugate_refl} and Lemma \ref{lem:red5} imply that all generators of $F_i$ lie in the same coset of $H_i$ and thus we have $H_i=\left\langle \Fp_i,h \right\rangle$ and $G_i/H_i = \Z_2$ in these cases. If $F_i$ has type $\D_4$, then $h$ has order $3$ and we have $\overline{s}_1=\overline{s}_3=\overline{s}_4$. 
\[
			\bullet_{s_1}-\bullet_{s_2} < 
		  \begin{array}{l}
		    \bullet_{s_4} \\
		    \bullet_{s_3} \\
		  \end{array}
\]
If $\overline{s}_1=\overline{s}_2$ also holds, then we have again $H_i=\left\langle \Fp_i,h \right\rangle$ and $G_i/H_i = \Z_2$. Otherwise the group $\tilde{H}_i=H_i \cap F_i$ is generated by the conjugates of $s_1s_3$ and $s_1s_4$ and the group $H_i$ is generated by $\tilde{H}_i$ and $h$. We have $H_i=\Wp(\D_4)$ (, but $H_i\neq \Fp_i$) and the quotient group $G_i/H_i=F_i/\tilde{H}_i$ has the Coxeter diagram
\[
	\circ_{\overline{s}_1}-\circ_{\overline{s}_2}
\]
  \end{proof}
The preceding three propositions show that each triple $(G_i,H_i,F_i)$ induced by a reflection-rotation groups occurs in one of the cases described in Theorem~\ref{thm:reducible_pairs}. Moreover, it is easy to check that each triple $(G_{rr},M,W)$ occurring in Theorem~\ref{thm:reducible_pairs} satisfies the conclusion of this Theorem concerning the reflections in $W$ and the properties described in Remark~1.

As a corollary we record
\begin{corollary}\label{cor:cosets_of_reflection} The reflections in $(G_i/H_i,\overline{S})$ are precisely the cosets of reflections in $F_i$.
\end{corollary}

In order to describe the structure of the whole group $G$ we need the following lemmas.

\begin{lemma}\label{lem:red3} Let $\tau, \tau' \in G_i$ and $s \in G_j$ be reflections such that $\tau s\in G$. Then $\tau's\in G$, if and only if $\overline{\tau}=\overline{\tau}'$.
\end{lemma}
\begin{proof} If $\tau's\in G$ then $\tau'ss\tau =\tau' \tau$ is a rotation of the first kind in $G$ and thus $\tau'\tau \in H_i$, i.e. $\overline{\tau}=\overline{\tau}'$. On the other hand $\overline{\tau}=\overline{\tau}'$ implies $\tau=h\tau'$ for some $h\in H_i$ and thus $\tau's=h^{-1}\tau s \in G$.
  \end{proof}

\begin{lemma}\label{lem:red4} Let $s_1, s_2 \in F_i$, $s_3 \in F_j$ and $s_4 \in F_{j'}$, $i\neq j,j'$, be reflections such that $g=s_1s_3,g'=s_2s_4 \in G$. Then the following two implications hold.
\begin{enumerate}
\item $j=j' \Rightarrow \mathrm{ord}(\overline{s}_1\overline{s}_2)= \mathrm{ord}(\overline{s}_3\overline{s}_4)$.
\item $j\neq j' \Rightarrow \mathrm{ord}(\overline{s}_1\overline{s}_2)\leq 2$.
\end{enumerate}
\end{lemma}
\begin{proof}
$(i)$ Assume that $j=j'$ and set $m=\mathrm{ord}(\overline{s}_1\overline{s}_2)$ and $n=\mathrm{ord}(\overline{s}_3\overline{s}_4)$. If $\overline{s}_1=\overline{s}_2$ then Lemma \ref{lem:red3} implies that $\overline{s}_3=\overline{s}_4$ and thus $m=n=1$. Otherwise, $(gg')^n=(s_1s_2)^n(s_3s_4)^n = (s_1s_2)^n h$ for some $h\in H_j$ implies that $(s_1s_2)^n\in G$ is a rotation of the first kind contained in $H_i$ and therefore $m|n$. In the same way we obtain $n|m$ and thus $m=n$.

$(ii)$ Since $(s_1s_2)^2=(gg')^2$ is a rotation of the first kind in $G$ or trivial, we deduce that $\mathrm{ord}(\overline{s}_1\overline{s}_2)\leq 2$.
  \end{proof}

For reflections $s\in F_i$ and $\tau \in F_j$ we call $\overline{s}$ and $\overline{\tau}$ \emph{related} if $s\tau \in G$ and $s \notin G$. Lemma \ref{lem:red3} shows that this notion is well-defined. For a Coxeter group $C$ we denote the set of reflections contained in $C$ by $X(C)$ and we set $\tilde{G}=G_1/H_1 \times \dots \times G_k/H_k$, $X=X(\tilde{G})$ and $X_i=X(G_i/H_i)$.

\begin{lemma}\label{lem:8.4}
Relatedness of reflections defines an equivalence relation on the set $X$ such that related reflections belong to different components. 
\end{lemma}
\begin{proof}
Let $\overline{s}_1 \in G_i/H_i$, $\overline{s}_2 \in G_j/H_j$ and $\overline{s}_3 \in G_l/H_l$ be reflections. If $\overline{s}_1$ and $\overline{s}_2$ are related as well as $\overline{s}_2$ and $\overline{s}_3$, then so are $\overline{s}_1$ and $\overline{s}_3$, because of $s_1s_3=(s_1s_2)(s_2s_3)\in G_i$. For $i=j$ the cosets $\overline{s}$ and $\overline{\tau}$ are related if and only if $\overline{s}=\overline{\tau}$.
  \end{proof}

For $i,j \in I$, $i\neq j$, we define $X_{ij}$ to be the set of reflections in $G_i/H_i$ that are related to reflections in $G_j/H_j$. Let $\Gamma_i$ be the Coxeter diagram of $G_i/H_i$ and set $\Gamma = \bigcup \Gamma_i$. The vertices of $\Gamma_i$ correspond to a set of simple reflections of $G_i/H_i$ (cf. \cite[p.~29]{MR1066460}).

\begin{lemma}\label{lem:8.5}
A reflection $\overline{s}$ in $G_i/H_i$ that is not related to any other reflection, corresponds to an isolated vertex of $\Gamma_i$.
\end{lemma}
\begin{proof}
Suppose that $\overline{s} \in G_i/H_i$ is a reflection not related to any other reflection and that $\overline{\tau} \in G_i/H_i$ is another reflection with $\mathrm{ord}(\overline{s}\overline{\tau})\geq 3$. Then $\overline{\tau}$ is related to some reflection $\overline{\tau}' \in G_j/H_j$ for some $j\neq i$, because otherwise we would have $s \tau \in H_i$ by Lemma \ref{lem:red2}. This implies $\mathrm{ord}(\overline{s}\overline{\tau}) \leq 2$ as in the proof of Lemma \ref{lem:red4}, $(ii)$, and thus the claim follows by contradiction.
  \end{proof}

\begin{lemma}\label{lem:red9}
Let $M$ be a connected component of $\Gamma_i$ and let $\mathcal{M}$ be the set of generators of $G_i/H_i$ that correspond to the vertices of $M$. If there exists a reflection in $\mathcal{M}$ related to another reflection, then there exists some $j\neq i$ such that $\mathcal{M} \subset X_{ij}$. Moreover, $\mathcal{M} \subset X_{ij} \cap X_{ik}$ for some $k\neq i,j$ only if $M = \circ$.
\end{lemma}
\begin{proof} The first claim follows from Lemma \ref{lem:red2} and the preceding lemma. Suppose we have distinct reflections $\overline{s},\overline{\tau} \in \mathcal{M}$ with $\mathrm{ord}(\overline{s}\overline{\tau}) \geq 3$. Again by Lemma \ref{lem:red2} and the preceding lemma there are $j,k$ such that $\overline{s}\in X_{ij}$ and $\overline{\tau}\in X_{ik}$. Lemma \ref{lem:red4}, $(ii)$ implies that $j=k$.
  \end{proof}
Due to this lemma the reflections related to the reflections of a nontrivial irreducible component of $\tilde{G}$ belong to a common $G_i/H_i$. The next proposition sharpens this statement.
\begin{proposition}
The set of nontrivial irreducible components of $\tilde{G}$ decomposes into pairs of isomorphic constituents that belong to different $G_i/H_i$ and for each such pair relatedness of reflections defines an isomorphism between its constituents that maps reflections onto related reflections.
\end{proposition}
\begin{proof}Let $M$ be a nontrivial connected component of $\Gamma_i$ and let $\mathcal{M}$ be the set of simple reflections corresponding to the vertices of $M$. According to Lemma \ref{lem:red9} there exists a unique $j\neq i$ such that $\mathcal{M} \subset X_{ij}$. Define $\varphi: \mathcal{M} \To G_j/H_j$ by mapping a generator $\overline{s}_i \in  \mathcal{M}$ to its related reflection in $G_j/H_j$. Due to Lemma \ref{lem:red4}, $(i)$, this map can be extended to a homomorphism $\varphi: \left\langle \mathcal{M} \right\rangle \To G_j/H_j$. We claim that the image $\varphi(\overline{s})$ of any reflection $\overline{s} \in \left\langle \mathcal{M} \right\rangle$ is a reflection related to $\overline{s}$. Since $\overline{s}$ is conjugate to a reflection in $M$ its image $\varphi(\overline{s})$ is a reflection in $G_j/H_j$ and thus a coset of a reflection in $F_j$, say $\varphi(\overline{s})=\overline{\tau}$ for some $\tau \in F_j$ (cf. Corollary \ref{cor:cosets_of_reflection}). Write $\overline{s}=\overline{s}_{i_1}\cdots \overline{s}_{i_l}$ for generators $\overline{s}_{i_1},\ldots, \overline{s}_{i_l} \in \mathcal{M}$ and let $\overline{\tau}_{i_j}=\varphi(\overline{s}_{i_j})$ be the related reflection. Then we have $\overline{\tau}=\overline{\tau}_{i_1}\cdots \overline{\tau}_{i_l}$. There exist $h_i \in H_i$ and $h_j \in H_j$ such that $s=h_i s_1\cdots s_l$ and $\tau=h_j \tau_1\cdots \tau_l$ and thus $s\tau = h_i h_j s_{i_1}\tau_{i_1}\cdots s_{i_l}\tau_{i_l} \in G$. Hence the reflections $\overline{s}$ and $\overline{\tau}=\varphi(\overline{s})$ are related.

The fact that $\varphi$ maps reflections onto related reflections together with Lemma \ref{lem:red4}, $(i)$ implies that $\varphi(\left\langle \mathcal{M} \right\rangle)$ is contained in an irreducible component of $G_j/H_j$ (cf. the argument below). Let $N$ be the connected component of $\Gamma_j$ such that $\left\langle \varphi(\mathcal{M}) \right\rangle \subseteq \left\langle \mathcal{N} \right\rangle$ where $\mathcal{N}$ is the set of simple reflections corresponding to the vertices of $N$. Since $N$ is connected, for $\overline{\tau}_0 \in \mathcal{N}$ there exist reflections $\overline{s}_k \in \mathcal{M}$ and $\overline{\tau}_0,\overline{\tau}_1,\ldots,\overline{\tau}_k \in \left\langle \mathcal{N} \right\rangle$ with $\overline{\tau}_k = \varphi(\overline{s}_k)$ and $\mathrm{ord}(\overline{\tau}_l\overline{\tau}_{l+1})\geq 3$, $l=0,\ldots, k-1$. Therefore, according to Lemma \ref{lem:red2}, Lemma \ref{lem:8.5} and Lemma \ref{lem:red4}, $(ii)$, there are reflections $\overline{s}_0,\ldots, \overline{s}_{k-1} \in G_i/H_i$ such that $\overline{s}_l$ and $\overline{\tau}_l$ are related for $l=0,\ldots, k-1$. Lemma \ref{lem:red4}, $(i)$, implies that $\mathrm{ord}(\overline{s}_l\overline{s}_{l+1})=\mathrm{ord}(\overline{\tau}_l\overline{\tau}_{l+1})\geq 3$, $l=0,\ldots, k-1$, and thus $\overline{s}_0, \ldots, \overline{s}_k\in \left\langle \mathcal{M} \right\rangle$. In particular, we have $\overline{\tau}_0 = \varphi(\overline{s}_0 ) \in \left\langle \varphi(\mathcal{M}) \right\rangle$ by what has been shown above and hence $\left\langle \varphi(\mathcal{M}) \right\rangle=\left\langle \mathcal{N} \right\rangle$, i.e. $\varphi$ is an epimorphism between the irreducible component $\left\langle \mathcal{M} \right\rangle$ of $G_i/H_i$ and the irreducible component $\left\langle \mathcal{N} \right\rangle$ of $G_j/H_j$. By the same argument there exists a homomorphism from $\left\langle \mathcal{N} \right\rangle$ to $G_i/H_i$ which maps $\left\langle \mathcal{N} \right\rangle$ onto $\left\langle \mathcal{M} \right\rangle$. Therefore, $\left\langle \mathcal{M} \right\rangle$ and $\left\langle \mathcal{N} \right\rangle$ have the same cardinality and thus $\varphi : \left\langle \mathcal{M} \right\rangle \To \left\langle \mathcal{N} \right\rangle$ is an isomorphism of Coxeter groups. 
  \end{proof}

Now we can prove Theorem \ref{thm:clas_red_prg}.

\begin{proof}[Proof of Theorem \ref{thm:clas_red_prg}]
Let $G$ be a reflection-rotation group and $\tilde{G}$ be given as above. According to what has been shown so far, relatedness of reflections induces an equivalence relation on the set of irreducible components of $\tilde{G}$ such that two related components belong to different $G_i/H_i$ (cf. Lemma \ref{lem:8.4}). By Lemma \ref{lem:red2} each $G_i/H_i$ contains at most one trivial irreducible component that is not related to another component. By the preceding proposition each equivalence class of a nontrivial irreducible component of $\tilde{G}$ contains precisely two isomorphic components and an isomorphism between them is induced by relatedness of reflections. Conversely, a family of possible triples $\{(G_i,H_i,F_i)\}_{i\in I}$ and an equivalence relation on the irreducible components of $\tilde{G}=G_1/H_1\times\dots\times G_k/H_k$ satisfying the conditions in Theorem \ref{thm:clas_red_prg} together with isomorphisms between the equivalent nontrivial irreducible components of $\tilde{G}$ that map reflections onto reflections defines a reflection-rotation group as described in the Introduction.

It remains to show that these assignments are inverse to each other. If we start with a reflection-rotation group $G$, assign to it a set of data as in the theorem and to this set of data another reflection-rotation group $\tilde{G}$, then $\tilde{G}$ is generated by the rotations in $G$ and thus coincides with $G$. Suppose we start with a set of data as in the theorem, including a family $\{(G_i,H_i,F_i)\}_{i\in I}$ of triples occurring in Theorem \ref{thm:reducible_pairs}, assign to it a reflection-rotation group $G$ and to this reflection-rotation group another set of data including a family of triples $\{(\tilde{G}_i,\tilde{H}_i,\tilde{F}_i)\}_{i\in J}$. Then we clearly have $I=J=\{1,\ldots,k\}$ and $G_i=\tilde{G}_i$ for all $i\in I$. We also have $H_i<\tilde{H}_i$ and $F_i<\tilde{F}_i$. By construction (cf. condition $(ii)$ in Theorem \ref{thm:clas_red_prg}) the quotient $G/(H_1\times \cdots \times H_k)$ does not contain nontrivial cosets of rotations of the first kind in $G$ and thus $H_i=\tilde{H}_i$. Since each reflection in $G_i$ is contained in $F_i$ (cf. Theorem \ref{thm:reducible_pairs}) $F_i=\tilde{F}_i$ holds as well. Now it is clear that the two sets of data coincide and thus the theorem is proven.
  \end{proof}
We record the following two corollaries. Recall that a reflection-rotation group is called \emph{indecomposable} if it cannot be written as a product of nontrival subgroups that act in orthogonal spaces (cf. Section \ref{sub:reducible_examples}).
\begin{corollary}\label{lem:ps_firstkind_only}
Let $G$ be a reducible reflection-rotation group that only contains rotations of the first kind. Then $G$ is a direct product of indecomposable rotation groups.
\end{corollary}
\begin{corollary}\label{lem:ps_secondkind_only}
For an indecomposable reflection-rotation group $G$ that does not contain rotations of the first kind one of the following three cases holds
\begin{enumerate}
	\item $k=2$, $\dim V_1 =\dim V_2$ and $G_1 \cong G_2$ for irreducible reflection groups $G_1$, $G_2$.
	\item $k>2$, $\dim V_1 =\ldots = \dim V_k=1$ and $G$ consists of all elements that change the sign of an even number of coordinates, i.e. $G=\Wp(\A_1\times \dots \times \A_1)$.
	\item $G=W(\A_1)$.
\end{enumerate}
\end{corollary}
Note that the group $G$ in case $(i)$ is only determind by $G_1$ \emph{and} the choice of an isomorphism between $G_1$ and $G_2$ that maps reflections onto reflection (cf. Section \ref{sub:reducible_examples}). 

\section{Isotropy groups of reflection-rotation groups}\label{sec:isotropy}

Isotropy groups of real reflection groups are generated by the reflections they contain \cite[Thm. 1.12 (c), p.~22]{MR1066460}. More generally, the same statement is true for isotropy groups of unitary reflection groups due to a theorem of Steinberg \cite[Thm.~1.5, p.~394]{MR0167535}. Independent proofs for Steinberg's theorem were given by Bourbaki \cite[Chapt. V, Exercise 8]{MR1080964} and Lehrer \cite{MR2052515}. Consequently, isotropy groups of rotation groups which are either unitary reflection groups considered as real groups or orientation preserving subgroups of real reflection groups are generated by the rotations they contain. In fact, it can be shown that isotropy groups of reflection-rotation groups are always generated by the reflections and rotations they contain. However, the only proof of this statement known to the authors uses the classification of reflection-rotation groups (cf. e.g. Lemma \ref{lem:isotropyL}). We end with the following question.

\begin{qst} Is it possible to prove that isotropy groups of rotation or reflection-rotation groups are generated by the reflections and rotations they contain \emph{without} using the classification?
\end{qst}

\section*{Notation and Tables}
\label{sec:notation}

We summarize the classification of irreducible reflection-rotation groups. There are $29$ primitive rotation groups that are a unitary reflection group considered as a real group. Among them $19$ groups occur in dimension $4$ and are listed under number $4-22$ in \cite[Chapt. 6 and Appendix D, Table 1]{MR2542964}. The remaining $10$ groups are generated by unitary reflections of order $2$ and are denoted as $W(\mathcal{J}_3^{(4)})$, $W(\mathcal{J}_3^{(5)})$, $W(\mathcal{K}_5)$, $W(\mathcal{K}_6)$, $W(\mathcal{L}_4)$, $W(\mathcal{M}_3)$, $W(\mathcal{M}_3)$, $W(\mathcal{N}_4)$, $W(\mathcal{O}_4)$ \cite[Chapt. 6 and Appendix D, Table 2]{MR2542964} (cf. Section \ref{sub:unitary}). All other primitive irreducible rotation groups are absolutely irreducible and are listed in Table \ref{tab:prim_abso_rot}.

\begin{table}[ht]
\begin{tabular}{ r | l } 
symbol & meaning \\
\hline
	$\cyc_n$, $\dih_n$ & Cyclic and dihedral group of order $n$ and $2n$, respectively \\
	$\sym_n$, $\alt_n$ & Symmetric and alternating group on $n$ letters. \\
	$\mathbf{G} $ & Preimage of a group $G<\SOr_3$ under the covering $\psi:\SU_2 \To \SOr_3$  \\
	 & (cf. Section \ref{sub:ps_low_dim} for the meaning of $(\mathbf{L}/\mathbf{L}_K;\mathbf{R}/\mathbf{R}_K)$  \\
		$\Wp$ & Orientation preserving subgroup of a real reflection group $W$. \\                      
	$\Ws$ & Unique extension of $\Wp$ by a normalizing rotation.\\
	$\Wt$ & Unique extension of $W$ by a normalizing rotation.\\
	$\mathfrak{P}$ & Plane system (cf. Introduction and Section \ref{sub:excep_subgroups}).\\
	$\Pp_5,\ldots,\Tt_8$ & Plane systems of certain type (cf. Section \ref{sub:excep_subgroups}).\\
	$M(\mathfrak{P})$ & Rotation group generated by rotations corresponding to the planes of $\mathfrak{P}$.\\
	$\Mt$ & Unique extension of a rotation group $M$ by a normalizing reflection.\\
	$R_n(G)$ & Image of the unique irreducible representation of $G$ in $\SOr_n$.\\
		$L$ & Rotation group in the normalizer of \\
		&  $N_{\SOr_8} (W(\I_2(4)) \otimes W(\I_2(4)) \otimes W(\I_2(4))) $ (cf. Section \ref{sec:group_L_sec}).\\
	$R_1$, $R_2$ & Root systems of type $E_8$ (cf. Section \ref{sec:group_L_sec}).\\
		$D(G)$ & Diagonal subgroup of a monomial group $G$. \\
	$D(n)$ & $D(W(\BC_n))$ \\
	$D^+(n)$ & $D(\Wp(\BC_n))$ \\
	
 \end{tabular}
\caption{List of notations.}
\label{tab:notations} 
\end{table}

\begin{table}[ht]
\begin{tabular}{ |r | l | l | l |}
\hline                       
  & G & description & order  \\
  \hline
  \hline  
 
 1.& $\Wp(\A_n)$ &  &$(n+1)!/2$  \\
 	\hline
 	2.&  $\Wp(\Hn_3)$&  & $2^2\cdot3\cdot5 = 60$  \\
  \hline 
  3.&&$(\mathbf{D}_{3m}/\mathbf{D}_{3m};\mathbf{T}/\mathbf{T})$  & $144m$  \\
 4.&&$(\mathbf{D}_{m}/\mathbf{D}_{m};\mathbf{O}/\mathbf{O})$  & $96m$  \\
 5.&&$(\mathbf{D}_{3m}/\mathbf{C}_{2m};\mathbf{O}/\mathbf{V})$  & $48m$  \\
 6.&&$(\mathbf{D}_{m}/\mathbf{C}_{2m};\mathbf{O}/\mathbf{T})$  & $48m$  \\
 7.&&$(\mathbf{D}_{2m}/\mathbf{D}_{m};\mathbf{O}/\mathbf{T})$  & $96m$  \\
 8.&&$(\mathbf{D}_{m}/\mathbf{D}_{m};\mathbf{I}/\mathbf{I})$  & $240m$  \\

 9.&&$(\mathbf{T}/\mathbf{T};\mathbf{O}/\mathbf{O})$  & $2^6\cdot 3^2=576$  \\
 10.&&$(\mathbf{T}/\mathbf{T};\mathbf{I}/\mathbf{I})$  & $2^5\cdot 3^2\cdot5=1440$  \\
 11.&&$(\mathbf{O}/\mathbf{O};\mathbf{I}/\mathbf{I})$  & $2^6\cdot 3^2\cdot5=2880$  \\
  12.&$\Wp(\A_4)$&$(\mathbf{I}/\mathbf{C}_1;\mathbf{I}/\mathbf{C}_1)^*$& $2^2\cdot3\cdot5 =60$  \\     
 13.&$\Ws(\A_4)$&$(\mathbf{I}/\mathbf{C}_2;\mathbf{I}/\mathbf{C}_2)^*$& $2^3\cdot3\cdot5 =120$  \\ 
 14.&$\Ws(\D_4)$&$(\mathbf{T}/\mathbf{T};\mathbf{T}/\mathbf{T})$  & $2^5\cdot 3^2=288$  \\
 15.& $\Wp(\F_4)$ &$(\mathbf{O}/\mathbf{T};\mathbf{O}/\mathbf{T})$ & $2^6\cdot 3^2=576$  \\
 16.&  $\Ws(\F_4)$&$(\mathbf{O}/\mathbf{O};\mathbf{O}/\mathbf{O}) $  & $2^7\cdot3^2=1152$  \\
 17.&$\Wp(\Hn_4)$&$(\mathbf{I}/\mathbf{I};\mathbf{I}/\mathbf{I})$  & $2^5\cdot3^2\cdot5^2=7200$  \\
 \hline 
 18.& $M(\Rr_5)$ & $R_5(\alt_5)$  & $2^2\cdot3\cdot5=60$  \\
 19.&  $\Ws(\A_5)$&   & $2^3\cdot3\cdot5^2=720$  \\
 \hline 
 20.&$M(\Ss_6)$ & $R_6(\PSL_2(7))$ & $2^3\cdot3\cdot7=168$  \\
 21.&  $\Wp(\E_6)$&   & $2^{6}\cdot 3^4 \cdot 5   = 25920$  \\
 22.&  $\Ws(\E_6)$&   & $2^{7}\cdot 3^4 \cdot 5  = 51840$  \\
 \hline
 23.&  $\Wp(\E_7)$&  & $2^{9}\cdot 3^4 \cdot 5 \cdot 7 = 1451520$  \\
 \hline
 24.& $M(\Tt_8)$&$L=W(R_1)\cap W(R_2)$  & $2^{13}\cdot 3^2 \cdot 5 \cdot 7= 2580480$  \\
 25.& $\Wp(\E_8)$&  & $2^{13}\cdot 3^5 \cdot 5^2 \cdot 7=348364800$  \\
 \hline
 \end{tabular}
\caption{Primitive absolutely irreducible rotation groups. Note that interchanging the left and right entry in $(\mathbf{L}/\mathbf{L}_K;\mathbf{R}/\mathbf{R}_K)$ yields isomorphic but nonconjugate groups in $\Or(4)$. See Table \ref{tab:notations} for unknown notations.}
\label{tab:prim_abso_rot} 
\end{table}

The imprimitive irreducible rotation groups that preserve a complex structure are induced by unitary reflection groups of type $G(m,p,n)$ (cf. Section \ref{sub:unitary}). All other imprimitive irreducible rotation groups are absolutely irreducible and are listed in Table \ref{tab:imprimitive_rot}. The groups $\Gs(km,k,n)$, $k=1,2$, are extensions of $G(km,k,n)$ by a normalizing rotation $\tau$ of the form $\tau(z_1,z_2,z_3\ldots,z_l)=(\overline{z}_1,\overline{z}_2,z_3\ldots,z_l)$ (cf. Lemma \ref{prp:class_imprimit_psg}). The groups $\Gs(km,k,2)_{\varphi}$ are described in Proposition \ref{prp:class_imprimit_psg}. 

\begin{table}[ht]
\begin{tabular}{ |r | l | l|l |}
\hline                       
  & G &description& order \\
  \hline 
  \hline 

 1.& $\Gs(km,k,2)_{\varphi}$& cf. Proposition \ref{prp:class_imprimit_psg} & $4km^2$  \\
 2.& $\Gs(km,k,l)$& $\left\langle G(km,k,l),\tau \right\rangle$, $k=1,2$, $l>2$, $km\geq 3$ & $2^{l-1} k^{l-1} m^l l!$  \\
 \hline  
 3.& $\Wp(\BC_n)$&  & $2^{n-1}n!$  \\ 
 4.& $\Wp(\D_n)$&   & $2^{n-2}n!$  \\ 
\hline 
 5.& $M(\Pp_5)$&$M_5=\Dp(5) \rtimes H_5$   & $2^4\cdot |H_5|=160$  \\
 \hline 
 6.& $M(\Pp_6)$&$M_6=\Dp(6) \rtimes H_6$    & $2^5\cdot |H_6|=1920$  \\
 \hline 
 7.& $M(\Qq_7)$&$M^p_7=\left\langle g_5,H_7  \right\rangle<M_7$, $|D(M^p_7)|=2^3$  & $2^3\cdot |H_7|=1344$  \\
 8.& $M(\Pp_7)$&$M_7=\Dp(7) \rtimes H_7$   & $2^6\cdot |H_7|=10752$  \\
 \hline 
 9.& $M(\Qq_8)$&$M^p_8=\left\langle g_5,H_8  \right\rangle<M_8$, $|D(M^p_8)|=2^4$ &$2^4\cdot |H_8|=21504$  \\
 10.& $M(\Pp_8)$&$M_8=\Dp(8) \rtimes H_8$   & $2^7\cdot |H_8|=172032$  \\
 \hline 
 \end{tabular}
\caption{Imprimitive absolutely irreducible rotation groups. See Table \ref{tab:permutation_Not} and Table \ref{tab:notations} for unknown notations.}
\label{tab:imprimitive_rot} 
\end{table}

All irreducible reflection-rotation groups that contain a reflection are listed in Table \ref{tab:ref-rot4}. The groups $\Gt (km,k,l)$, $k=1,2$, are generated by $G(km,k,l)$ and a reflection of type $s(z_1,\ldots, z_l)=(\overline{z}_1,z_2,\ldots, z_l)$.

\begin{table}[ht]
\begin{tabular}{ |r | l | l|l|} 
\hline                      
  & G &description& order \\
  \hline 
	 \hline 
	 1.& $W$ & any irreducible reflection group &  \\
  \hline 
  2.& $\Gt (km,k,l)$ & $\left\langle G(km,k,l),s \right\rangle$, $k=1,2$, $l\geq2$, $km\geq 3$ & $2^l k^{l-1} m^l l!$  \\
  \hline
  3.& $\Mt(\D_n)$ & $D(n) \rtimes \alt_n$  & $2^{n-1}n!$  \\
  \hline
 4.& $\Wt(\A_4)$&   & $2^4\cdot3\cdot5=240$  \\
 5.& $\Wt(\D_4)$&     & $2^6\cdot 3^2=576$  \\
 6.& $\Wt(\F_4)$&    & $2^8\cdot3^2=2304$  \\
 \hline
 7.& $\Mt(\Pp_5)$ & $\Mt_5=D(5) \rtimes H_5$  & $2^5\cdot |H_5|=320$  \\
 8.& $\Wt(\A_5)$&    & $2^5\cdot 3^2\cdot 5=1440$  \\
 \hline
  9.& $\Mt(\Pp_6)$ & $\Mt_6=D(6) \rtimes H_6$   & $2^6\cdot |H_6|=3840$  \\
  10.& $\Wt(\E_6)$   &    & $2^{8}\cdot 3^4 \cdot 5=103680$  \\ 
\hline
 11.& $\Mt(\Pp_7)$ &  $\Mt_7=D(7) \rtimes H_7$   & $2^7\cdot |H_7|=21504$  \\
 \hline
 12.& $\Mt(\Pp_8)$ & $\Mt_8=D(8) \rtimes H_8$  & $2^8\cdot |H_8|=344064$  \\
\hline 
 \end{tabular}
\caption{Irreducible reflection-rotation groups that contain a reflection. See Table \ref{tab:permutation_Not} and Table \ref{tab:notations} for unknown notations. }
\label{tab:ref-rot4} 
\end{table}

\begin{table}[ht]
\begin{tabular}{ r | l } 
symbol & meaning \\
\hline

	$H_5$ & $\left\langle (1,2)(3,4),(2,3)(4,5)\right\rangle \subgr  \sym_5$, $H_5\cong \dih_5$ \\
	$H_6$&$\left\langle (1,2)(3,4),(1,5)(2,3),(1,6)(2,4)\right\rangle \subgr  \sym_6$, $H_6\cong \alt_5$ \\
	$H_7$&$\left\langle g_1,g_2,g_3\right\rangle \subgr  \sym_7$, $H_7\cong \PSL_2(7)\cong \SL_3(2)$   \\                   
	$H_8$&$\left\langle g_1,g_2,g_3,g_4\right\rangle\subgr  \sym_8$, $H_8\cong \AG_3(2)\cong \Z_2^3 \rtimes \SL_3(2)$. \\
	$g_i$& $g_1=(1,2)(3,4),\; g_2=(1,5)(2,6),\; g_3=(1,3)(5,7),\; g_4=(1,2)(7,8) \; g_5 = (1,\overline{2})(3,\overline{4})$ \\
	$(i,\overline{j})$ & Linear transformation that maps $e_i$ to $-e_j$, $-e_j$ to $e_i$ and $e_k$ to $e_k$ for $k\neq i,j$,\\& where $e_1,\ldots,e_n$ are standard basis vectors.  \\
	
 \end{tabular}
\caption{Explanation of symbols appearing in Table \ref{tab:imprimitive_rot} and Table \ref{tab:ref-rot4}.}
\label{tab:permutation_Not} 
\end{table}

\end{document}